\pdfoutput=1

\documentclass[a4paper,twoside]{article}
\usepackage{import}

\usepackage{amssymb}
\usepackage{amsmath}
\usepackage{amsthm}

\usepackage{float}
\usepackage{tabularx}
\usepackage{longtable}

\usepackage{graphicx}
\usepackage{subfig}
\usepackage{epstopdf}
\DeclareGraphicsExtensions{.pdf, .eps, .png, .jpg}
\usepackage[ngerman,english]{babel}
\usepackage{color}
\usepackage{cite}
\usepackage{comment} %zum auskommentieren von Textstellen

\usepackage{fancyhdr}
\usepackage{datetime}
\newtheorem{problem}{Problem}

\newtheorem{remark}{Remark}

\newtheorem{lemma}{Lemma}
\newtheorem{theorem}{Theorem}

\newtheorem{example}{Example}

\newcommand{\keywords}{\textbf{Key words.}  }

\title{{\scshape Residual-type a posteriori estimators\\ for a singularly perturbed reaction-diffusion\\ variational inequality\\ \vspace{-0.3em} -\\ \vspace{-0.3em} reliability, efficiency and robustness}}

\author{Mirjam Walloth\\*[3\baselineskip]
\small Fachbereich Mathematik, TU Darmstadt \\
\small  Dolivostra{\ss}e 15, 64293 Darmstadt, {\it walloth@mathematik.tu-darmstadt.de}}
\date{}

\begin{document}
\maketitle

\begin{abstract}
We derive a new residual-type a posteriori estimator for a singularly perturbed reaction-diffusion problem with obstacle constraints. 
It generalizes robust residual estimators for unconstrained singularly perturbed equations. Upper and lower bounds are derived with respect to an error notion which measures the error of the solution and of a suitable approximation of the constraining force in a $\epsilon$-dependent energy norm and its dual norm. While the robust upper bound is proven for non-discrete obstacle function, the local lower bounds are derived for discrete obstacle functions. For the proof of the local lower bounds we construct special bubble functions which cope with the structure of the approximation of the constraining force and the $\epsilon$-dependency. 
\end{abstract}

\keywords{ singularly perturbed reaction diffusion, obstacle problem, residual-type a posteriori estimator}

\section{Introduction}

In this work we present a residual-type a posteriori estimator for a singularly perturbed reaction-diffusion problem with obstacle constraints.
An example of this can be found in phase-field models for fracture propagation with irreversibility constraints, see e.g. \cite{Miehe_Hofacker_Welschinger_2010}.

A posteriori estimators are widely used in the numerical simulation to obtain informations about the approximation quality of the discrete solution as usually the exact solution is unknown. This information can be further used for mesh adaptation to improve the quality of the discrete solution for given computational resources.

It is important that the estimator is reliable and efficient, i.e. constitutes upper and lower bounds to the error at least up to so-called oscillation terms. Thus, the estimator is equivalent to the error which implies that the error is neither over- nor underestimated. A reliable and efficient estimator for linear elliptic problems, which is attractive in view of its simplicity, is the standard residual estimator \cite{Verfuerth_2013}. 
Regarding singularly perturbed equations which depend on a parameter $\epsilon<<1$ it is important that the estimator is robust, i.e. the constants in the upper and lower bounds are independent of $\epsilon$. Otherwise the equivalence relation between error and estimator is destroyed for $\epsilon\rightarrow 0$. For singularly perturbed reaction-diffusion equations a robust residual estimator has been presented in \cite{Verfuerth_1998b}.

Standard residual estimators yield an upper but no lower bounds for obstacle problems. In \cite{Veeser_2001} a first efficient and reliable residual-type a posteriori estimator has been derived. Therein the error was measured in both unknowns, the solution and a suitable approximation of the constraining forces. This idea has been adapted and improved for different obstacle and contact problems. e.g. \cite{Moon_Nochetto_Petersdorff_Zhang_2007, Krause_Veeser_Walloth_2015, Walloth_2018}.

As far as we know, we derive the first efficient and reliable residual-type a posteriori estimators for singularly perturbed reaction-diffusion problems with obstacle constraints. The estimator reduces to the robust residual estimator of  \cite{Verfuerth_1998b} if no contact occurs. 

In order to measure the error in both unknowns and to deal with the aspect of robustness we define an energy norm depending on $\epsilon$ for the error in the solution and a corresponding dual norm for the error of a suitable approximation of the constraining force, called quasi-discrete constraining force. The definition of this approximation reflecting the local structure is important for the efficiency as well as the localization of the estimator contributions. 
A key ingredient in the derivation of efficient and reliable residual-type a posteriori estimators for problems with constraints is the Galerkin functional which replaces the linear residual of elliptic unconstrained problems. The derivation of the robust upper bound consists of deriving an upper bound of the epsilon-dependent dual norm of the Galerkin functional and an upper bound of a duality pairing between the error in the solution and the constraining force. Besides estimator contributions known from \cite{Verfuerth_1998b} a complementarity residual is part of the estimator. 
For discrete gap functions the estimator contributions vanish in the so-called area of full-contact. Thus, the estimator perceives that adaptive refinement cannot improve the solution if it is fixed to the obstacle. 

We give local lower bounds of all estimator contributions in the case of discrete gap functions. Therefore we define a  linear combination of bubble functions with special properties. The definition has to cope with the structure of the quasi-discrete constraining force as these auxiliary functions relate the local estimator contributions to the Galerkin functional and thus to the error measure. 
In contrast to  \cite{ Krause_Veeser_Walloth_2015} where a similar ansatz has been used, the interior residual of a reaction-diffusion problem which is part of the estimator cannot be approximated by a constant interior residual plus oscillation terms. Further, to tackle the epsilon-dependency of the estimator and the error measure we use bubble functions on modified elements like in  \cite{Verfuerth_1998b}.
 
In the last section we provide numerical results. We analyze numerically the adaptively refined grids, the convergence of the error and the robustness of the estimator. 
Further, we show what happens if one uses the standard residual estimator not designed for a constrained problem and a residual-type estimator which is suited for the obstacle problem but not for the epsilon-dependency.

\section{The singularly perturbed reaction-diffusion \- variational inequality}

In this section we present the weak and the discrete problem formulations for the singularly perturbed reaction-diffusion problem with obstacle constraints. 

\subsection{Weak formulation}
The domain is denoted by $\Omega\subset\mathbb{R}^d$, $d=2,3$ and the boundary by $\Gamma=\partial\Omega$ which is subdivided in the Neumann boundary $\Gamma^N$ and the Dirichlet boundary $\Gamma^D$. 

The solution space of the weak formulation is the subset
\begin{equation*}
\mathcal{H}:= \{\psi\in H^1(\Omega)\mid \mathrm{tr}|_{\Gamma^D}(\psi) = \varphi^D \}
\end{equation*}
of $H^1(\Omega)$ where $\mathrm{tr}$ is the trace operator. For convenience in the discrete approximation of the Dirichlet values we assume $\varphi_D$ to be continuous and piecewise linear on $\Gamma^D$.
In the following we omit the special notation for the trace operator. The space of test functions is given by
$\mathcal{H}_0:= \{\psi \in H^1(\Omega)\mid \mathrm{tr}|_{\Gamma^D}(\varphi) = 0 \}$
and its dual is $\mathcal{H}^{\ast}$.
For the obstacle function $g\in H^{1}(\Omega)$ with $\varphi^D\leq g$ on $\Gamma_D$ we define the admissible set
\begin{equation}\label{ContinuousAdmissibleSet}
\mathcal{K}:=\{\psi\in \mathcal{H}\mid \psi\leq g\}.
\end{equation}
We assume the force density $f$ and the Neumann data $\pi$ to be $L^2$-functions on $\Omega$ or $\Gamma^N$, respectively.  The $L^2$-norm and its scalar product are denoted by $\|\cdot\|$ and $\left<\cdot,\cdot\right>$ without any subindex.  The duality pairing between $H^1$ and its dual $H^{-1}$ is given by $\left<\cdot,\cdot\right>_{-1,1}$ and the corresponding norms are $\|\cdot\|_1$ and $\|\cdot\|_{-1}$.  Later on, we need restrictions to subdomains which are indicated by a further subindex, e.g., $\|\cdot\|_{1,\omega}$ for $\omega\subset\Omega$.

Finally, we define the symmetric bilinear form 
\begin{align}\label{BilinearForm_a}
a_{\epsilon}(\zeta,\psi) := \left<\epsilon^2 \nabla \zeta,\nabla \psi \right> +\left<\zeta , \psi \right>\quad\forall\zeta,\psi\in H^1,
\end{align} 
with $\epsilon<<1$.
Thus, the weak formulation of the singularly perturbed reaction-diffusion problem is given by Problem \ref{Problem:WeakForm}.
\begin{problem}{Weak formulation}\label{Problem:WeakForm}

We seek a solution $\varphi\in\mathcal{K}$ such that
\begin{equation}\label{WeakContinuous}
a_{\epsilon}(\varphi,\varphi-\psi )\leq\left< f, \varphi-\psi\right> + \left<\pi,\varphi-\psi\right>_{\Gamma^N}\quad\forall \psi\in \mathcal{K}.
\end{equation}
\end{problem}
The unique solvability of Problem \ref{Problem:WeakForm} is given by the Theorem of Lions and Stampacchia.
It exists a distribution $\lambda\in\mathcal{H}^{\ast}$ which turns the variational inequality (\ref{WeakContinuous}) in an equation
\begin{equation*}
\left<\lambda ,\psi\right>_{-1,1}:=\left<f,\psi \right> + \left<\pi,\psi\right>_{\Gamma^N} -a_{\epsilon}(\varphi, \psi )\quad \forall \psi\in \mathcal{H}_0.
\end{equation*}
It is called Lagrange multiplier or constraining force density. Due to the variational inequality the constraining force fulfills the following sign condition 
\begin{equation}\label{SignConditionContinuous}
\left<\lambda,\varphi-\psi\right>_{-1,1}\ge 0.
\end{equation}

\subsection{Finite element formulation}

In the discrete setting we assume the domain $\Omega$ to be polygonal. The mesh $\mathfrak{M}$, resolving the domain, consists of simplicial elements $\mathfrak{e}\in\mathfrak{M}$ which are either disjoint or share a node $p$, an edge or a face $\mathfrak{s}$. The
polygonal boundary segments $\Gamma^D, \Gamma^N$ are resolved by the mesh, too, meaning that their boundaries
$\partial\Gamma^N,\partial\Gamma^D$ are either nodes or edges. The set of nodes $p$ is given by $\mathfrak{N}$ and we distinguish between the set $\mathfrak{N}^D$ of nodes on the Dirichlet boundary, the set $\mathfrak{N}^N$ of nodes
at the Neumann boundary and the set of interior nodes $\mathfrak{N}^I$. The mesh is taken from a shape-regular family, meaning that the ratio of the diameter of any element to the diameter of its inscribed circle is uniformly bounded. 

Further, we define a patch $\omega_p$ as the interior of the union of all elements sharing the node $p$. We call the union of all sides in the interior of $\omega_p$, not including the boundary of $\omega_p$ skeleton and denote it by $\gamma_p^I$. For Neumann boundary nodes we denote the intersections between $\Gamma$ and $\partial\omega_p$ by  $\gamma_p^N:=\Gamma^N\cap\partial\omega_p$.
Further, we will make use of $\omega_{\mathfrak{s}}$ which is the union of all elements sharing a side $\mathfrak{s}$.

The solution space of the weak formulation is the subset of the linear finite element space with incorporated Dirichlet values $\varphi^D$ denoted by
\begin{equation*}
\mathcal{H}_{\mathfrak{m}}:= \{\psi_{\mathfrak{m}} \in \mathcal{C}^0(\bar{\Omega})\mid\forall\mathfrak{e}\in\mathfrak{M}, \;\psi_{\mathfrak{m}}|_{\mathfrak{e}}\in \mathbb{P}_1(\mathfrak{e}) \text{ and }\psi_{\mathfrak{m}}=\varphi^D\;\mbox{on }\Gamma^D\}
\end{equation*}
and the space of test functions is given by
\begin{equation*}
\mathcal{H}_{\mathfrak{m},0}:= \{\psi_{\mathfrak{m}} \in \mathcal{C}^0(\bar{\Omega})\mid\forall\mathfrak{e}\in\mathfrak{M}, \;\psi_{\mathfrak{m}}|_{\mathfrak{e}}\in \mathbb{P}_1(\mathfrak{e}) \text{ and }\psi_{\mathfrak{m}}= 0\;\mbox{on }\Gamma^D\}.
\end{equation*}
The nodal basis functions of the finite element spaces are denoted by $\phi_p$. 

Let $g_{\mathfrak{m}}$ be a linear finite element approximation of the obstacle function $g$ with $g_{\mathfrak{m}}\geq \varphi^D$ on $\Gamma_D$, then the discrete admissible set is given by 
\begin{equation}\label{DiscreteAdmissibleSet_ptwise}
\mathcal{K}_{\mathfrak{m}}:=\{\psi_{\mathfrak{m}}\in \mathcal{H}_{\mathfrak{m}}\mid \psi_{\mathfrak{m}}\leq g_{\mathfrak{m}}\}.
\end{equation}
If $g=g_{\mathfrak{m}}$, $\mathcal{K}_{\mathfrak{m}}\subset\mathcal{K}$ holds. 

The discrete problem formulation is given in Problem \ref{Problem:DiscreteForm}.
\begin{problem}{Discrete formulation}\label{Problem:DiscreteForm}\\
Find $\varphi_{\mathfrak{m}}\in\mathcal{K}_{\mathfrak{m}}$ fulfilling the variational inequality
\begin{equation*}
a_{\epsilon}(\varphi_{\mathfrak{m}}, \varphi_{\mathfrak{m}}-\psi_{\mathfrak{m}}) \leq \left<f,\varphi_{\mathfrak{m}}-\psi_{\mathfrak{m}}\right> + \left<\pi,\varphi_{\mathfrak{m}}-\psi_{\mathfrak{m}}\right>_{\Gamma^N}\quad\forall\psi_{\mathfrak{m}}\in\mathcal{K}_{\mathfrak{m}}.
\end{equation*}
\end{problem}
The unique solvability of Problem \ref{Problem:DiscreteForm}  follows just as in the continuous case from the Theorem of Lions and Stampacchia. 

Proceeding as in the continuous case we can define the discrete constraining force density 
\begin{equation}\label{DiscreteConstrainingForce}
\left< \lambda_{\mathfrak{m}},\psi_{\mathfrak{m}} \right>_{-1,1} := \left<f,\psi_{\mathfrak{m}}\right>+\left<\pi,\psi_{\mathfrak{m}}\right>_{\Gamma^N}-a_{\epsilon}(\varphi_{\mathfrak{m}}, \psi_{\mathfrak{m}}) \quad\forall \psi_{\mathfrak{m}}\in\mathcal{H}_{\mathfrak{m},0}.
\end{equation}
Later on, when we use integration by parts, we need the definition of the jump term $[\nabla\psi_{\mathfrak{m}}]:= \nabla|_{\mathfrak{e}}\psi_{\mathfrak{m}}\cdot\boldsymbol{n}- \nabla|_{\tilde{\mathfrak{e}}}\psi_{\mathfrak{m}}\cdot\boldsymbol{n}$ where $\mathfrak{e}, \tilde{\mathfrak{e}}$ are neighbouring elements and $\boldsymbol{n}$ ist the unit outward normal on the common side of the two elements.

\section{Main results}\label{Sec:Preliminaries} 
This section is devoted to the formulation of the main results of this article while the proofs will be given in Sections \ref{Sec:Reliability} and \ref{Sec:Efficiency}.
After defining the quasi-discrete constraining force, the Galerkin functional and the error measure, we define the estimator contributions and formulate the Theorems of reliability and efficiency.

\subsection{Quasi-discrete constraining force}

The discrete constraining force $\lambda_{\mathfrak{m}}$ (\ref{DiscreteConstrainingForce}) equals the definition of the linear residual $\mathcal{R}_{\mathfrak{m}}^{lin}$ for linear elliptic equations. We recall that in the derivation of standard residual estimators for linear equations the residual plays an important role as it is equivalent to the error \cite{Verfuerth_2013}. In the case of variational inequalities this equivalence is disturbed as the linear residual is related to the discrete constraining force as well as to the error. Thus, an error estimator based on the linear residual would overestimate the error. 

In \cite{Veeser_2001} the first efficient and reliable residual-type a posteriori estimator for a variational inequality was proposed. Therein, the error was measured in both unknowns the solution $\varphi$ and the constraining force $\lambda$. 
In order to compare the continuous and discrete constraining forces, we cannot simply take $\lambda_{\mathfrak{m}}$ as by definition (\ref{DiscreteConstrainingForce}) it is a functional on the space of discrete functions and not a functional on $\mathcal{H}_0$. 
There is no unique definition how $\lambda_{\mathfrak{m}}$ acts on a function in $\mathcal{H}_0$ which is not in $\mathcal{H}_{\mathfrak{m},0}$. Thus, we have to define an approximation of $\lambda$ as a function on $\mathcal{H}_0$ on the basis of the properties of the discrete solution $\varphi_{\mathfrak{m}}$ and $\lambda_{\mathfrak{m}}$.
Later on we call this approximation quasi-discrete constraining force.
In \cite{Veeser_2001} such an approximation as a functional on $H^1$ has been proposed by means of lumping $\sum_{p\in\mathfrak{N}^C}s_p\phi_p$,
where $s_p = \frac{\left<\lambda_{\mathfrak{m}},\phi_p\right>_{-1,1}}{\int_{\omega_p}\phi_p}\ge 0$
are the node values of the lumped discrete constraining force. The sign condition follows from the discrete variational inequality. 
As the lumped discrete constraining force  is a discrete function a complementarity condition, i.e. $\sum_{p\in\mathfrak{N}^C}s_p\phi_p\cdot(\varphi_{\mathfrak{m},1}-g_{\mathfrak{m}})=0$, cannot be fulfilled in the so-called semi-contact zone which consists of elements having nodes which are in contact and nodes which are not in contact. It is only valid in so-called full-contact areas where $\varphi_{\mathfrak{m},1}=g_{\mathfrak{m}}$  and in non-actual-contact areas where $\varphi_{\mathfrak{m},1}<g_{\mathfrak{m}}$. 

Especially for the efficiency and the localization of a posteriori error estimation it is very advantageous, if such a quasi-discrete constraining force,  can be defined differently for the different areas of full- and semi-contact to reflect local properties. Such an approach has been first used for the derivation of an a posteriori error estimator in \cite{Fierro_Veeser_2003} and applied to obstacle and contact problems in \cite{Nochetto_Siebert_Veeser_2005, Moon_Nochetto_Petersdorff_Zhang_2007, Krause_Veeser_Walloth_2015, Walloth_2018}. 
Following this approach  we distinguish between full-contact nodes $p\in\mathfrak{N}^{fC}$ and semi-contact nodes $p\in\mathfrak{N}^{sC}$. Full-contact nodes are those nodes for which the solution is fixed to the obstacle $\varphi_{\mathfrak{m}}=g_{\mathfrak{m}}$ on $\omega_p$ and the sign condition is fulfilled $\left<\mathcal{R}^{lin}_{\mathfrak{m}},\varphi \right>_{-1,1,\omega_p}\ge0 $ $\forall \varphi\ge 0\in \mathcal{H}_0(\omega_p)$. The latter condition means that the solution is locally not improvable, see the explanation in \cite{Moon_Nochetto_Petersdorff_Zhang_2007}. Semi-contact nodes are those nodes for which $\varphi_{\mathfrak{m}}(p)=g_{\mathfrak{m}}(p)$ holds but not the conditions of full-contact. 
Based on this classification we define the quasi-discrete constraining force
\begin{align}\label{QuasiDiscreteConstrainingForce}
&\left< \tilde{\lambda}_{\mathfrak{m}},\psi \right>_{-1,1} := \sum_{p\in\mathfrak{N}^{sC}}\left< \tilde{\lambda}_{\mathfrak{m}}^p,\psi\phi_p \right>_{-1,1} + \sum_{p\in\mathfrak{N}^{fC}}\left< \tilde{\lambda}_{\mathfrak{m}}^p,\psi\phi_p \right>_{-1,1}\quad\forall\psi\in\mathcal{H}_0.
\end{align}
For semi-contact nodes
\begin{align*}
\left< \tilde{\lambda}_{\mathfrak{m}}^p,\psi\phi_p \right>_{-1,1}
&:= \left<\lambda_{\mathfrak{m}} ,\phi_p \right>_{-1,1}c_p(\psi)\\
& =\int_{\gamma_p^I}\epsilon^2[\nabla \varphi_{\mathfrak{m}}]^Ic_p(\psi)\phi_p  + \int_{\gamma_p^N}(\pi-\epsilon^2\nabla \varphi_{\mathfrak{m}}\cdot\boldsymbol{n})c_p(\psi)\phi_p\\
&\qquad + \int_{\omega_p} (f -  \varphi_{\mathfrak{m}})c_p(\psi)\phi_p 
\end{align*}
holds with $c_p(\psi)=0$ for $p\in\mathfrak{N}^D$ and $c_p(\psi)= \frac{\int_{\tilde{\omega}_p}\psi\phi_p}{\int_{\tilde{\omega}_p}\phi_p}$, otherwise, where $\tilde{\omega}_p$ is the patch around $p$ with respect to two uniform red-refinements. 
For full-contact nodes
\begin{align*}
\left< \tilde{\lambda}_{\mathfrak{m}}^p,\psi\phi_p \right>_{-1,1}
& :=  \left<\mathcal{R}^{lin}_{\mathfrak{m}}, \psi\phi_p\right>_{-1,1}\\
& :=\int_{\gamma_p^I}\epsilon^2[\nabla \varphi_{\mathfrak{m}}]\psi\phi_p  + \int_{\gamma_p^N}(\pi-\epsilon^2\nabla \varphi_{\mathfrak{m}}\cdot\boldsymbol{n})\psi\phi_p\\
&\qquad + \int_{\omega_p} (f  -  \varphi_{\mathfrak{m}})\psi\phi_p 
\end{align*}
holds. 

\subsection{Error measure and Galerkin functional}
Corresponding to the bilinear form $a_{\epsilon}(\cdot,\cdot)$ given in (\ref{BilinearForm_a}) we define the energy norm
\begin{equation}\label{EnergyNorm}
\|\cdot\|_{\epsilon}:= \left\{\epsilon^2\|\nabla \cdot\|^2 + \|\cdot\|^2 \right\}^{\frac{1}{2}},
\end{equation}
compare \cite{Verfuerth_1998b}. 
As we aim to measure the error in both unknowns $\varphi$ and $\lambda$ we define further the dual norm $\|\cdot\|_{\ast,\epsilon}:= \frac{\mathrm{sup}_{\psi\in H^1}\left<\cdot, \psi \right>_{-1,1}}{\|\psi\|_{\epsilon}}$.

The error measure we consider for the derivation of the estimator is given by
\begin{equation}\label{ErrorMeasure}
\|\varphi_{\mathfrak{m}}-\varphi\|_{\epsilon} + \|\lambda-\tilde{\lambda}_{\mathfrak{m}}\|_{\ast,\epsilon}.
\end{equation}
Accordingly, the linear residual which is used in the derivation of a posteriori estimators for linear elliptic equations is replaced by a so-called Galerkin functional which takes into account the errors in both unknowns
 \begin{align}\label{DefGalerkinFunctional_a}
 \left<G_{\mathfrak{m}},\psi\right>_{-1,1} &:= a_{\epsilon}(\varphi-\varphi_{\mathfrak{m}},\psi) + \left<\lambda-\tilde{\lambda}_{\mathfrak{m}},\psi\right>_{-1,1}\quad\forall\psi\in\mathcal{H}_0
 \end{align}
 and can be reformulated as follows
 \begin{align}\nonumber
  \left<G_{\mathfrak{m}},\psi\right>_{-1,1} &=\left<f,\psi\right> + \left<\pi, \psi\right> - a_{\epsilon}( \varphi_{\mathfrak{m}}, \psi) - \left< \tilde{\lambda}_{\mathfrak{m}},\psi \right>_{-1,1}\\ \nonumber
&= \sum_{p\in\mathfrak{N}\backslash \mathfrak{N}^{fC}} \int_{\gamma_p^I}\epsilon^2[\nabla \varphi_{\mathfrak{m}}](\psi-c_p(\psi))\phi_p\\ \nonumber
&\qquad + \int_{\gamma_p^N}(\pi-\epsilon^2\nabla \varphi_{\mathfrak{m}}\cdot\boldsymbol{n})(\psi-c_p(\psi))\phi_p\\  \label{DefGalerkinFunctional_b}
&\qquad +\int_{\omega_p} (f  -  \varphi_{\mathfrak{m}})(\psi-c_p(\psi))\phi_p.  
 \end{align}

The relation between the dual norm of the Galerkin functional $\|G_{\mathfrak{m}}\|_{\ast,\epsilon}$ and the error measure (\ref{ErrorMeasure}) follows from
\begin{eqnarray}\label{LowerBound0}
\|G_{\mathfrak{m}}\|_{\ast,\epsilon}\lesssim \|\varphi-\varphi_{\mathfrak{m}}\|_{\epsilon}+\|\lambda-\tilde{\lambda}_{\mathfrak{m}}\|_{\ast,\epsilon},
\end{eqnarray}
and
\begin{align}\label{UpperBound0}
\|\varphi-\varphi_{\mathfrak{m}}\|^2_{\epsilon}\leq \|G_{\mathfrak{m}}\|^2_{\ast,\epsilon} + 2\left< \tilde{\lambda}_{\mathfrak{m}}-\lambda,\varphi-\varphi_{\mathfrak{m}} \right>_{-1,1},
\end{align}
and
\begin{align}\label{UpperBound1}
\|\lambda-\tilde{\lambda}_{\mathfrak{m}}\|^2_{\ast,\epsilon}\leq 2\left(\|G_{\mathfrak{m}}\|^2_{\ast,\epsilon}+ \|\varphi-\varphi_{\mathfrak{m}}\|^2_{\epsilon}\right),
\end{align}
compare \cite[Lemma 3.4]{Veeser_2001}.

\subsection{Error estimator and main results}\label{Subsec:EstimatorMainRes}
The error estimator 
\begin{equation}\label{Def_Estimator}
\eta := \sum_{k=1}^7\eta_{k}
\end{equation}
for which we prove reliability and efficiency in Sections  \ref{Sec:Reliability} and \ref{Sec:Efficiency} consists of the following local contributions
\begin{align*}
\eta_{1}:=&\left(\sum_{p\in\mathfrak{N}\backslash\mathfrak{N}^{fC}}\eta^2_{1,p}\right)^{\frac{1}{2}}, & \eta_{1,p}:=&\mathrm{min}\{\frac{h_p}{\epsilon},1\}\|f  -  \varphi_{\mathfrak{m}}\|_{\omega_p}\\
\eta_{2}:=&\left(\sum_{p\in\mathfrak{N}\backslash\mathfrak{N}^{fC}}\eta^2_{2,p}\right)^{\frac{1}{2}}, &\eta_{2,p}:=& \mathrm{min}\{\frac{h_p}{\epsilon},1\}^{\frac{1}{2}}\epsilon^{-\frac{1}{2}}\|\epsilon^2[\nabla \varphi_{\mathfrak{m}}]\|_{\gamma_p^I}\\
\eta_{3}:=&\left(\sum_{p\in\mathfrak{N}\backslash\mathfrak{N}^{fC}}\eta^2_{3,p}\right)^{\frac{1}{2}}, &\eta_{3,p}:=&\mathrm{min}\{\frac{h_p}{\epsilon},1\}^{\frac{1}{2}}\epsilon^{-\frac{1}{2}}\|\pi-\epsilon^2\nabla \varphi_{\mathfrak{m}}\cdot\boldsymbol{n}\|_{\gamma_p^N}\\ \displaybreak
\eta_{4}:=&\left(\sum_{p\in\mathfrak{N}^{sC}}\eta^2_{4,p}\right)^{\frac{1}{2}}, &\eta_{4,p}:=&\left(s_p \int_{\tilde{\omega}_p}(g_{\mathfrak{m}}-\varphi_{\mathfrak{m}})\phi_p\right)^{\frac{1}{2}}\\
\eta_{5}:=&\left(\sum_{p\in\mathfrak{N}^{sC}}\eta^2_{5,p}\right)^{\frac{1}{2}}, &\eta_{5,p}:= & \left(s_p \int_{\tilde{\omega}_p}((g-g_{\mathfrak{m}})^{+})\phi_p\right)^{\frac{1}{2}}\\
\eta_{6}:=&\left(\sum_{p\in\mathfrak{N}^{fC}}\eta^2_{6,p}\right)^{\frac{1}{2}}, &\eta_{6,p}:= &\left(\left<\tilde{\lambda}^p_{\mathfrak{m}} , (g-g_{\mathfrak{m}})^{+}\phi_p \right>_{-1,1}\right)^{\frac{1}{2}}\\
\eta_{7}:=&\left(\sum_{p\in\mathfrak{N}^C}\eta^2_{7,p}\right)^{\frac{1}{2}}, &\eta_{7,p}:= &\| (\varphi_{\mathfrak{m}} -g)^+\phi_p\|_{\omega_p,\epsilon}
\end{align*}
where $s_p = \frac{\left<\lambda_{\mathfrak{m}},\phi_p\right>_{-1,1}}{\int_{\omega_p}\phi_p}$, $h_p:=\mathrm{diam}(\omega_p)$ and $\zeta^+:=\mathsf{max}\{\zeta,0\}$ denotes the positive part of a function. We emphasize that the estimator contributions related to the constraints only contribute to $\eta$ in the area where the constraints are active, actually only in the area of semi-contact if $g=g_{\mathfrak{m}}$. This property is called localization of estimator contributions. 
In the absence of any contact, we have $\eta_{k,p}=0$ for $k=4,\ldots,7$ such that $\eta$ equals the standard residual estimator 
\begin{equation}\label{StdResEst}
\eta^{std}:= \sum_{k=1}^3\eta_k,
\end{equation}
for singularly perturbed reaction-diffusion equations, see \cite{Verfuerth_1998b}.

In Section \ref{Sec:Reliability}  we prove that $\eta$ constitutes a robust upper bound where robust means that the constant in the bound does not depend on $\epsilon$ such that the validity of the estimator holds for arbitrary choices of $\epsilon$.
\begin{theorem}{\bf Reliability of the error estimator}\label{Theorem:UpperBound}\\
The error estimator $\eta$ provides a robust upper bound of the error measure (\ref{ErrorMeasure}):
\begin{equation*}
\|\varphi_{\mathfrak{m}}-\varphi\|_{\epsilon} + \|\lambda-\tilde{\lambda}_{\mathfrak{m}}\|_{\ast,\epsilon}\lesssim \eta.
\end{equation*}
\end{theorem} 

In order to formulate the local lower bounds we denote by $\bar{f}$ and $\bar{\pi}$ the piecewise constant approximations of $f$ and $\pi$ and we abbreviate $\mathrm{osc}_p(f):=\mathrm{min}\{\frac{h_p}{\epsilon},1\} \|\bar{f}-f\|_{\omega_p}$ and $\mathrm{osc}_p(\pi):=  \mathrm{min}\{\frac{h_p}{\epsilon},1\}^{\frac{1}{2}}\epsilon^{-\frac{1}{2}}\|\bar{\pi}-\pi\|_{\gamma_p^N}$.
In Section \ref{Sec:Efficiency} we derive the local lower bounds which are summarized in the following Theorems
\begin{theorem}{\bf Local lower bounds by $\eta_{1,p},\eta_{2,p},\eta_{3,p}$}\label{Theorem:LowerBound}\\
The error estimator contributions $\eta_{k,p}$, $k=1,2,3$ constitute the following robust local lower bounds
\begin{equation*}
\eta_{k,p}\lesssim\|\varphi-\varphi_{\mathfrak{m}}\|_{\epsilon,\omega_p} + \|\lambda-\tilde{\lambda}_{\mathfrak{m}}\|_{\ast,\epsilon,\omega_p} + \mathrm{osc}_p(f) + \mathrm{osc}_p(\pi).
\end{equation*}
\end{theorem}
\begin{theorem}{\bf Local lower bound by $\eta_{4,p}$}\label{Theorem:LowerBound2}\\
For nodes $p\in\mathfrak{N}^{sC}$  we have the robust local lower bound
\begin{equation}\label{LowerBound4_Robust}
\begin{split}
\eta_{4,p} \lesssim& \|\varphi-\varphi_{\mathfrak{m}}\|_{\epsilon,\omega_p} + \|\lambda-\tilde{\lambda}_{\mathfrak{m}}\|_{\ast,\epsilon,\omega_p} +\mathrm{osc}_p(f) + \mathrm{osc}_p(\pi)\\
&  + \mathrm{min}\{\frac{h_p}{\epsilon},1\}^{\frac{1}{2}}\epsilon^{-\frac{1}{2}}\|\epsilon^2[\nabla g_{\mathfrak{m}}]\|_{\gamma^I_p}.
\end{split}
\end{equation}
Otherwise, for nodes $p\in\mathfrak{N}^{sC}$ with $h_p> \epsilon$ we have the local lower bound
\begin{equation}\label{LowerBound4_NonRobust}
\begin{split}
\eta_{4,p} \leq & C(\epsilon^{-\frac{3}{2}}) \left(\|\varphi-\varphi_{\mathfrak{m}}\|_{\epsilon,\omega_p} + \|\lambda-\tilde{\lambda}_{\mathfrak{m}}\|_{\ast,\epsilon,\omega_p} +\mathrm{osc}_p(f) + \mathrm{osc}_p(\pi)\right .\\
&\left. + \epsilon^{-\frac{1}{2}} \|\epsilon^2[\nabla g_{\mathfrak{m}}]\|_{\gamma^I_p} \right).
\end{split}
\end{equation}
\end{theorem}
\begin{remark}
We note that the additional term $\mathrm{min}\{\frac{h_p}{\epsilon},1\}^{\frac{1}{2}}\epsilon^{-\frac{1}{2}}\|\epsilon^2[\nabla g_{\mathfrak{m}}]\|_{\gamma^I_p}$ in Theorem \ref{Theorem:LowerBound2} only occurs for $p\in\mathfrak{N}^{sC}$.  
The local lower bound (\ref{LowerBound4_Robust}) for $h_p\leq \epsilon$ shows that the decay of $\eta_{4,p}$ is of the same order as the other estimator contributions. 
In the application we expect the semi-contact zone to be well resolved, especially with respect to $\epsilon$ such that $h_p\leq \epsilon$ after a finite number of adaptive refinement steps such that the local lower bound is robust everywhere. 
 \end{remark}
 We do not provide lower bounds in terms of the additional error estimator contributions $\eta_{k,p}$ for $k=5,6,7$, depending on data approximation but notice that they cannot be neglected in the upper bound because all the other estimator contributions might be zero, while the real problem is not resolved due to $g_{\mathfrak{m}}\neq g$, compare examples 4.1 and 4.2 in \cite{Moon_Nochetto_Petersdorff_Zhang_2007} and subsection 6.2 in \cite{Krause_Veeser_Walloth_2015}.\\

In Section \ref{Sec:NumRes} we present numerical experiments to show the benefits of the new estimator $\eta$. Amongst others we will compare it
to a variant $\eta^{nr}$ which can be easily derived without taking care of the aspect of robustness. 
Imagine a residual-type a posteriori estimator for Problem \ref{Problem:WeakForm} would be derived with respect to the $H^1$-norm of the error, not 
paying attention to the $\epsilon$-dependency. The derivation would basically follow along the lines of Section \ref{Sec:Reliability} and \ref{Sec:Efficiency}. The proofs would be simplified as there is no need to care about the $L^2$-approximation for an energy norm and the standard definitions of the bubble functions can be used. 
The estimator contributions $\eta_k$ change in the way that $\mathrm{min}\{\frac{h_p}{\epsilon},1\}$ is replaced by  $h_p$ and  $\mathrm{min}\{\frac{h_p}{\epsilon},1\}^{\frac{1}{2}}\epsilon^{-\frac{1}{2}}$ is replaced by $h_p^{\frac{1}{2}}$.

\section{Reliability of the error estimator}\label{Sec:Reliability}
Based on the combination of (\ref{UpperBound0}) and (\ref{UpperBound1})
\begin{align}\label{UpperBoundAbstract}
\|\varphi-\varphi_{\mathfrak{m}}\|^2_{\epsilon}+\|\lambda-\tilde{\lambda}_{\mathfrak{m}}\|^2_{\ast,\epsilon}\leq 5\|G_{\mathfrak{m}}\|^2_{\ast,\epsilon} + 6\left< \tilde{\lambda}_{\mathfrak{m}}-\lambda,\varphi-\varphi_{\mathfrak{m}} \right>_{-1,1}
\end{align}
we derive the reliability of the estimator by first deriving a computable upper bound of $\|G_{\mathfrak{m}}\|^2_{\ast,\epsilon}$ and second of $\left< \tilde{\lambda}_{\mathfrak{m}}-\lambda,\varphi-\varphi_{\mathfrak{m}} \right>_{-1,1}$. Thus, the proof of Theorem \ref{Theorem:UpperBound} will follow from Lemma \ref{UpperBoundGalerkin} and Lemma \ref{LemmaComplementarityResidual}.

\subsection{Upper bound of Galerkin functional}

In this subsection we give the proof of 
\begin{lemma}\label{UpperBoundGalerkin}
The Galerkin functional defined in (\ref{DefGalerkinFunctional_a}) satisfies 
\begin{equation*}
\|G_{\mathfrak{m}}\|_{\ast,\epsilon}\lesssim \left(\sum_{k=1}^3\eta^2_k\right)^{\frac{1}{2}}
\end{equation*}
\end{lemma}
We will make use of the following results on the reference elements $\hat{\mathfrak{e}}$.
\begin{lemma}\label{LemmaBdEstimate}
Let $\hat{v}\in H^1(\hat{\mathfrak{e}})$ vanish on $\hat{\mathfrak{s}}_p$. Then the estimate 
\begin{equation*}
\|\hat{v}\|_{\hat{\mathfrak{s}}}\lesssim \|\hat{v}\|^{\frac{1}{2}}_{\hat{\mathfrak{e}}}\|\nabla \hat{v}\|^{\frac{1}{2}}_{\hat{\mathfrak{e}}}
\end{equation*}
holds for all other sides $\hat{\mathfrak{s}}\cap\hat{\mathfrak{e}}\neq \emptyset$. 
\end{lemma}
The proof follows along the lines of \cite[Lemma 3.2]{Verfuerth_1998b}.

\begin{lemma}[$L^2$-approximation with respect to energy norm]\label{LemmaL2Approx}
Let  $c_p(\psi)=\frac{\int_{\tilde{\omega}_p}\psi\phi_p}{\int_{\tilde{\omega}_p}\phi_p}$
with $\tilde{\omega}_p\subset\omega_p$ for all $\mathfrak{N}\backslash\mathfrak{N}^D$ and $c_p(\psi)=0$ for $p\in\mathfrak{N}^D$. Then the $L^2$-approximation properties with respect to the energy norm (\ref{EnergyNorm}) hold
\begin{align}\label{L2ApproxEnergyElement}
 \|(\psi-c_p(\psi))\phi_p\|_{\omega_p}&\lesssim \mathrm{min}\{\frac{h_p}{\epsilon},1\}\|\psi\|_{\epsilon,\omega_p}\\ \label{L2ApproxEnergySide}
 \|(\psi-c_p(\psi))\phi_p\|_{\mathfrak{s}}&\lesssim\mathrm{min}\{\frac{h_p}{\epsilon},1\}^{\frac{1}{2}}\epsilon^{-\frac{1}{2}}\|\psi\|_{\epsilon,\omega_{\mathfrak{s}}}.
 \end{align}
\end{lemma}
\begin{proof}
To get the $L^2$-approximation property for $c_p(\psi)=\frac{\int_{\tilde{\omega}_p}\psi\phi_p}{\int_{\tilde{\omega}_p}\phi_p}$
with $\tilde{\omega}_p\subset\omega_p$ set $\zeta := \psi-c_p(\psi)$ and define the constant $<c>:= \frac{\int_{\omega_p}\zeta}{\int_{\omega_p}1}$. 
Thus, 
\begin{align*}
\int_{\tilde{\omega}_p}\zeta\phi_p = \int_{\tilde{\omega}_p}\psi\phi_p - \frac{\int_{\tilde{\omega}_p} \psi\phi_p}{\int_{\tilde{\omega}_p}\phi_p}\int_{\tilde{\omega}_p}\phi_p = 0
\end{align*}
and by adding and subtracting $<c>$
\begin{align*}
\zeta = \zeta - \int_{\tilde{\omega}_p}\zeta\phi_p = \zeta -<c> - \frac{1}{\int_{\tilde{\omega}_p}\phi_p}\int_{\tilde{\omega}_p}(\zeta -<c>)\phi_p.
\end{align*}
Now, we take the $L^2$-norm, apply the triangle inequality, the Poincar\'e inequality with mean value zero, the Cauchy-Schwarz inequality and use the fact that $\tilde{\omega}_p$ is the patch around $p$ with respect to a twice uniformly red-refined mesh, such that the diameter is a fixed portion of $h_p= \mathrm{diam}(\omega_p)$.
\begin{align}\nonumber
\|\psi-c_p(\psi)\|_{\omega_p} &\leq \|\zeta\|_{\omega_p}\leq\|\zeta-<c>\|_{\omega_p} + \|\frac{1}{\int_{\tilde{\omega}_p}\phi_p}\int_{\tilde{\omega}_p}(\zeta -<c>)\phi_p\|_{\omega_p}\\ \nonumber
&\lesssim h_p\|\nabla \zeta\|_{\omega_p} + \|\zeta -<c>\|_{\omega_p}\\ \label{L2ApproxH1}
&\lesssim h_p \|\nabla\psi\|_{\omega_p}. 
\end{align}
 For $c_p(\psi)=0$ the $L^2$-approximation (\ref{L2ApproxH1}) follows directly from the Poincar\'e inequality $\|\psi\|_{\omega_p}\lesssim h_p\|\nabla \psi\|_{\omega_p}$.
Together with 
$ \|(\psi-c_p(\psi))\phi_p\|_{\omega_p}\lesssim \|\psi\|_{\omega_p}$ we deduce the $L^2$-approximation property with respect to the energy norm (\ref{L2ApproxEnergyElement}).

It remains to derive the $L^2$- approximation property for sides $\mathfrak{s}$. Therefore we fix a node $p$ and denote the sides which are opposite to $p$ in $\mathfrak{e}$ by $\mathfrak{s}_p$. We note that $v:=(\psi-c_p(\psi))\phi_p|_{\mathfrak{s}_p}=0$. Let $F_{\mathfrak{e}}$ be the transformation $F_{\mathfrak{e}}: \mathfrak{e}\rightarrow \hat{\mathfrak{e}}$ from the element $\mathfrak{e}$ on the reference element $\hat{\mathfrak{e}}$. Thus $\hat{v}:=v\circ F^{-1}_{\mathfrak{e}}|_{\hat{\mathfrak{s}}_p}=0$, too. 
 
We apply the transformation rule to $\|v\|$ and $\|\nabla v\|$ and the result of Lemma \ref{LemmaBdEstimate} to get
$$\|(\psi-c_p(\psi))\phi_p\|_{\mathfrak{s}}\lesssim \|(\psi-c_p(\psi))\phi_p\|^{\frac{1}{2}}_{\mathfrak{e}}\|\nabla ((\psi-c_p(\psi))\phi_p)\|^{\frac{1}{2}}_{\mathfrak{e}}.$$
Next we apply the product rule and triangle inequality 
\begin{align*}
\|\nabla((\psi-c_p(\psi))\phi_p)\|_{\mathfrak{e}}&\leq\|\nabla(\psi-c_p(\psi))\phi_p\|_{\mathfrak{e}} + \|(\psi-c_p(\psi))\nabla\phi_p\|_{\mathfrak{e}}\\
&\lesssim \|\nabla(\psi-c_p(\psi))\|_{\mathfrak{e}} + h_{\mathfrak{e}}^{-\frac{1}{2}}\|(\psi-c_p(\psi))\|_{\mathfrak{e}}.
\end{align*} 
Thus, we get together with the $L^2$-approximation property (\ref{L2ApproxEnergyElement}) on the elements and $\|\nabla\psi\|_{\omega_\mathfrak{s}}\leq\epsilon^{-1}\|\psi\|_{\epsilon,\omega_{\mathfrak{s}}}$ the $L^2$-approximation property on the sides
\begin{align}\nonumber
\|(\psi-c_p(\psi))\phi_p\|_{\mathfrak{s}}&\leq\|(\psi-c_p(\psi))\phi_p\|_{\omega_{\mathfrak{s}}} ^{\frac{1}{2}}\|\nabla((\psi-c_p(\psi))\phi_p)\|^{\frac{1}{2}}_{\omega_{\mathfrak{s}}} \\ \nonumber
&\lesssim h_{\mathfrak{e}}^{-\frac{1}{2}}\|(\psi-c_p(\psi))\|_{\omega_{\mathfrak{s}}} + \|(\psi-c_p(\psi))\|^{\frac{1}{2}}_{\omega_{\mathfrak{s}}}\|\nabla(\psi-c_p(\psi)) \|^{\frac{1}{2}}_{\omega_{\mathfrak{s}}}\\ \nonumber
&\lesssim h_{\mathfrak{e}}^{-\frac{1}{2}}\mathrm{min}\{\frac{h_{\mathfrak{e}}}{\epsilon},1\}\|\psi\|_{\epsilon,\omega_{\mathfrak{s}}} + \mathrm{min}\{\frac{h_{\mathfrak{e}}}{\epsilon},1\}^{\frac{1}{2}}\|\psi\|^{\frac{1}{2}}_{\epsilon,\omega_{\mathfrak{s}}}\epsilon^{-\frac{1}{2}}\|\psi\|^{\frac{1}{2}}_{\epsilon,\omega_{\mathfrak{s}}}\\ \nonumber 
&\lesssim\mathrm{min}\{\frac{h_{\mathfrak{e}}}{\epsilon},1\}^{\frac{1}{2}}\epsilon^{-\frac{1}{2}}\|\psi\|_{\epsilon,\omega_{\mathfrak{s}}}
\end{align}
\end{proof}
Together with this preliminary results we can give the proof of Lemma \ref{UpperBoundGalerkin}.
\begin{proof}[Proof of Lemma \ref{UpperBoundGalerkin}] 
In order to derive an upper bound of the dual norm of the Galerkin functional, we use the representation (\ref{DefGalerkinFunctional_b}) and Cauchy-Schwarz inequality 
 \begin{align}\nonumber
 \left<G_{\mathfrak{m}}, \psi\right>
&\leq
\sum_{p\in\mathfrak{N}\backslash \mathfrak{N}^{fC}} \|\epsilon^2[\nabla \varphi_{\mathfrak{m}}]\|_{\gamma_p^I}\|(\psi-c_p(\psi))\phi_p\|_{\gamma_p^I} \\
\nonumber
&\qquad +\|\pi-\epsilon^2\nabla \varphi_{\mathfrak{m}}\cdot\boldsymbol{n}\|_{\gamma_p^N}\|(\psi-c_p(\psi))\phi_p\|_{\gamma_p^N}\\ \label{BoundGalerkin}
&\qquad +\|f  -  \varphi_{\mathfrak{m}}\|_{\omega_p}\|(\psi-c_p(\psi))\phi_p\|_{\omega_p}.
 \end{align}
 
 Combining (\ref{BoundGalerkin}) and (\ref{L2ApproxEnergyElement}, \ref{L2ApproxEnergySide}) we get the bound of the dual norm of the Galerkin functional 
\begin{align*}
\left<G_{\mathfrak{m}}, \psi \right>_{-1,1}&\lesssim \left(\sum_{p\in\mathfrak{N}\backslash \mathfrak{N}^{fC}}\left( \mathrm{min}\{\frac{h_{\mathfrak{e}}}{\epsilon},1\}^{\frac{1}{2}}\epsilon^{-\frac{1}{2}}\|\epsilon^2[\nabla \varphi_{\mathfrak{m}}]\|_{\gamma_p^I} \right.\right. \\
\nonumber
&\qquad \left.\left.+\mathrm{min}\{\frac{h_{\mathfrak{e}}}{\epsilon},1\}^{\frac{1}{2}}\epsilon^{-\frac{1}{2}}\|\pi-\epsilon^2\nabla \varphi_{\mathfrak{m}}\cdot\boldsymbol{n}\|_{\gamma_p^N}\right.\right.\\ 
&\qquad \left.\left.+\mathrm{min}\{\frac{h_{\mathfrak{e}}}{\epsilon},1\}\|f  -  \varphi_{\mathfrak{m}}\|_{\omega_p}\right)^2\right)^{\frac{1}{2}}\left(\sum_{p\in\mathfrak{N}}\|\psi\|^2_{\epsilon,\omega_p}\right)^{\frac{1}{2}}
\end{align*}
and thus 
\begin{align}\nonumber
\|G_{\mathfrak{m}}\|_{\ast,\epsilon}&= \frac{\mathrm{sup}_{\psi\in H^1}\left<G_{\mathfrak{m}}, \psi \right>_{-1,1}}{\|\psi\|_{\epsilon}}\lesssim \sum_{k=1}^3\eta_{k}.
\end{align}
\end{proof}

\subsection{Complementarity residual}
In this subsection we give the proof of 
\begin{lemma}[Complementarity residual]\label{LemmaComplementarityResidual}
Assume $\mathcal{K}_{\mathfrak{m}}\subset\mathcal{K}$ then 
\begin{align}\nonumber
 \left< \tilde{\lambda}_{\mathfrak{m}}-\lambda,\varphi-\varphi_{\mathfrak{m}} \right>_{-1,1}\lesssim  \eta^2_4
 \end{align}
 holds. Otherwise, in the case $\mathcal{K}_{\mathfrak{m}}\not\subset\mathcal{K}$
 \begin{align}\nonumber
\left< \tilde{\lambda}_{\mathfrak{m}}-\lambda,\varphi-\varphi_{\mathfrak{m}} \right>_{-1,1} \lesssim \frac{1}{2}\|\lambda -\tilde{\lambda}_{\mathfrak{m}} \|^2_{\ast,\epsilon} + \eta^2_{7} + \eta^2_4 + \eta^2_5 + \eta^2_6.
\end{align}
\end{lemma}
\begin{proof}
First, we consider the case $g = g_{\mathfrak{m}}$ such that $\varphi_{\mathfrak{m}}\in\mathcal{K}_{\mathfrak{m}}\subset\mathcal{K}$ is an admissible function in Problem \ref{Problem:WeakForm} so that $\left<-\lambda,\varphi-\varphi_{\mathfrak{m}} \right>_{-1,1}\leq 0.$ For semi-contact nodes we exploit the sign condition (\ref{SignConditionContinuous}) and for full-contact nodes $\varphi_{\mathfrak{m}}=g_{\mathfrak{m}}$ and $\left<\mathcal{R}^{lin}_{\mathfrak{m}} ,\psi \right>_{-1,1}\leq 0$ for $\psi\leq 0$. Further, we use $\varphi\leq g_{\mathfrak{m}}$ and
 we exploit that $\mathrm{diam}(\tilde{\omega}_p)$ is a fixed portion of $\mathrm{diam}(\omega_p)$. Thus, 
 \begin{align}\nonumber
 \left< \tilde{\lambda}_{\mathfrak{m}}-\lambda,\varphi-\varphi_{\mathfrak{m}} \right>_{-1,1}&\leq \left<\tilde{\lambda}_{\mathfrak{m}} , \varphi-\varphi_{\mathfrak{m}}\right>_{-1,1}\\ \nonumber
 & = \sum_{p\in\mathfrak{N}^{sC}}\left<\lambda_{\mathfrak{m}} ,\phi_p \right>_{-1,1}c_p(\varphi-\varphi_{\mathfrak{m}}) + \sum_{p\in\mathfrak{N}^{fC}}\left<\mathcal{R}^{lin}_{\mathfrak{m}} ,(\varphi-g_{\mathfrak{m}})\phi_p \right>_{-1,1}\\ \label{ComplementarityBoundDiscrete}
 &\lesssim \sum_{p\in\mathfrak{N}^{sC}}s_p\int_{\tilde{\omega}_p}(g_{\mathfrak{m}}-\varphi_{\mathfrak{m}})\phi_p = \eta^2_4.
 \end{align}
 
Second, we consider an arbitrary choice of $g\in H^{1}(\Omega)$. In this case $\mathcal{K}_{\mathfrak{m}}\not\subset\mathcal{K}$ and thus, we cannot exploit $\left<-\lambda,\varphi-\varphi_{\mathfrak{m}} \right>_{-1,1}\leq 0.$ Therefore, we define 
\begin{equation*}
\varphi^{\ast}_{\mathfrak{m}} := \mathrm{min}\{\varphi_{\mathfrak{m}}, g\}\in H^{1}(\Omega)
\end{equation*}
and make use of  $\left< \lambda,  \varphi^{\ast}_{\mathfrak{m}}-\varphi\right>_{-1,1}\leq 0$.
With this we get
\begin{align*}
\left< \lambda, \varphi_{\mathfrak{m}} -\varphi\right>_{-1,1} &= \left< \lambda, \varphi_{\mathfrak{m}} -\varphi^{\ast}_{\mathfrak{m}} + \varphi^{\ast}_{\mathfrak{m}}-\varphi\right>_{-1,1} \\
&\leq  \left< \lambda -\tilde{\lambda}_{\mathfrak{m}} , \varphi_{\mathfrak{m}} -\varphi^{\ast}_{\mathfrak{m}} \right>_{-1,1} + \left<\tilde{\lambda}_{\mathfrak{m}} , \varphi_{\mathfrak{m}} -\varphi^{\ast}_{\mathfrak{m}} \right>_{-1,1} \\
&\lesssim \frac{1}{2}\|\lambda -\tilde{\lambda}_{\mathfrak{m}} \|^2_{\ast,\epsilon} + \frac{1}{2}\| \varphi_{\mathfrak{m}} -\varphi^{\ast}_{\mathfrak{m}}\|^2_{\epsilon} + \left<\tilde{\lambda}_{\mathfrak{m}} , \varphi_{\mathfrak{m}} -\varphi^{\ast}_{\mathfrak{m}} \right>_{-1,1}
\end{align*}
and thus, exploiting, $(\varphi-\varphi_{\mathfrak{m}}^{\ast})\leq (g-\varphi_{\mathfrak{m}}^{\ast})\leq (g-\varphi_{\mathfrak{m}})^+ \leq (g-g_{\mathfrak{m}})^+ + (g_{\mathfrak{m}}-\varphi_{\mathfrak{m}}) $ and additionally for full-contact nodes $g_{\mathfrak{m}} = \varphi_{\mathfrak{m}}$ we get
\begin{align}\nonumber
\left< \tilde{\lambda}_{\mathfrak{m}}-\lambda,\varphi-\varphi_{\mathfrak{m}} \right>_{-1,1} &\lesssim \frac{1}{2}\|\lambda -\tilde{\lambda}_{\mathfrak{m}} \|^2_{\ast,\epsilon} + \frac{1}{2}\| \varphi_{\mathfrak{m}} -\varphi^{\ast}_{\mathfrak{m}}\|^2_{\epsilon} + \left<\tilde{\lambda}_{\mathfrak{m}} , \varphi -\varphi^{\ast}_{\mathfrak{m}} \right>_{-1,1}\\ \nonumber
&\lesssim \frac{1}{2}\|\lambda -\tilde{\lambda}_{\mathfrak{m}} \|^2_{\ast,\epsilon} + \frac{1}{2}\| \varphi_{\mathfrak{m}} -\varphi^{\ast}_{\mathfrak{m}}\|^2_{\epsilon}\\ \nonumber
&\qquad \quad + \sum_{p\in\mathfrak{N}^{sC}}\left< \lambda_{\mathfrak{m}},\phi_p\right>_{-1,1}c_p(\varphi-\varphi^{\ast}_{\mathfrak{m}}) \\ \nonumber
&\qquad\quad + \sum_{p\in\mathfrak{N}^{fC}}\left<\tilde{\lambda}^p_{\mathfrak{m}} ,(\varphi-\varphi^{\ast}_{\mathfrak{m}})\phi_p \right>_{-1,1}\\ \nonumber
&\lesssim  \frac{1}{2}\|\lambda -\tilde{\lambda}_{\mathfrak{m}} \|^2_{\ast,\epsilon} + \frac{1}{2}\| \varphi_{\mathfrak{m}} -\varphi^{\ast}_{\mathfrak{m}}\|^2_{\epsilon}\\ \nonumber
&\qquad \quad + \sum_{p\in\mathfrak{N}^{sC}}s_p\int_{\tilde{\omega}_p}(g_{\mathfrak{m}}-\varphi_{\mathfrak{m}})\phi_p +s_p\int_{\tilde{\omega}_p}((g-g_{\mathfrak{m}})^{+})\phi_p\\ \nonumber
&\qquad \quad + \sum_{p\in\mathfrak{N}^{fC}}\left<\tilde{\lambda}^p_{\mathfrak{m}} , (g-g_{\mathfrak{m}})^{+}\phi_p \right>_{-1,1}\\ \label{ComplementarityBound}
&=   \frac{1}{2}\|\lambda -\tilde{\lambda}_{\mathfrak{m}} \|^2_{\ast,\epsilon} + \eta^2_{7} + \eta^2_4 + \eta^2_5 + \eta^2_6.
\end{align}
\end{proof}

\section{Efficiency of the error estimator}\label{Sec:Efficiency}
In this Section we give the proofs of Theorem \ref{Theorem:LowerBound} and \ref{Theorem:LowerBound2}.

\subsection{Local error bound by $\eta_{1,p},\eta_{2,p},\eta_{3,p}$}
We start with $\eta_{1,p}$ for which we use the properties of the element bubble functions $\Psi_{\mathfrak{e}}:= c\Pi_{p\in\mathfrak{e}}\phi_p$,  see \cite{Verfuerth_2013}.
\begin{itemize}
\item $0\leq \Psi_{\mathfrak{e}}\leq 1$
\item $\|\nabla (\Psi_{\mathfrak{e}} v)\|_{\mathfrak{e}}\lesssim h_{\mathfrak{e}}^{-1}\|v\|_{\mathfrak{e}}$ for all polynomials $v$
\end{itemize}
where $h_{\mathfrak{e}}$ is the diameter of the element $\mathfrak{e}$.

In the following we make use of $h_p\approx h_{\mathfrak{e}}\approx h_{\mathfrak{s}}$ with $h_{\mathfrak{s}}=\mathrm{diam}(\omega_{\mathfrak{s}})$.
With respect to the energy norm (\ref{EnergyNorm}) this implies for all polynomials $v$
\begin{align}\label{WeightedInverseInequality}
\|\Psi_{\mathfrak{e}}v\|_{\epsilon,\mathfrak{e}}\lesssim (\epsilon h_{\mathfrak{e}}^{-1}+1)\|v\|_{\mathfrak{e}}\lesssim \mathrm{max}\{\frac{\epsilon}{h_{\mathfrak{e}}},1\}\|v\|_{\mathfrak{e}}.
\end{align}
For all $p\in\mathfrak{N}$,  $\tilde{\omega}_p$ is the patch around $p$ with respect to a twice uniformly red-refined mesh $\widetilde{\mathfrak{M}}$ with $\tilde{\mathfrak{e}}\in\widetilde{\mathfrak{M}}$ and $h_{\tilde{\mathfrak{e}}}= ch_{\mathfrak{e}}$. We define a linear combination of element bubble functions $\Psi_j$ with respect to all elements $\tilde{\mathfrak{e}}_j\subset \mathfrak{e}$, i.e. $\theta_{\mathfrak{e}} = \sum_{j=1} a_j\Psi_j$. We choose $a_j=0$ for all elements $\tilde{\mathfrak{e}}_j$ containing a node $p\in\mathfrak{N}^{sC}$ such that 
\begin{equation} \label{MeanValueZero1}
\int_{\tilde{\mathfrak{e}}_j}\phi_q\theta_{\mathfrak{e}}\phi_{p} = 0\quad\forall q\in\mathfrak{e}. 
\end{equation}

The other coefficients of the linear combination are chosen such that the bubble function $\theta_{\mathfrak{e}}$ fulfills the following six conditions 
\begin{align*}
\int_{\mathfrak{e}}\phi_q\phi_r=\sum_{p\in\mathfrak{N}\backslash\mathfrak{N}^{fC}}\int_{\mathfrak{e}} \phi_q\phi_r\theta_{\mathfrak{e}}\phi_p\quad \forall q, r \in \mathfrak{e}.
\end{align*}
As we have even more degrees of freedom (coefficients $a_j$) than conditions this problem is solvable. 

We abbreviate $r(\varphi_{\mathfrak{m}}):= f- \varphi_{\mathfrak{m}}$, define the linear approximation $\bar{r}(\varphi_{\mathfrak{m}}):= \bar{f}- \varphi_{\mathfrak{m}}$ where $\bar{f}$ is assumed to be the mean value.
Thus, 
\begin{align}\nonumber
&\|\bar{r}(\varphi_{\mathfrak{m}})\|_{\mathfrak{e}}^2\\ \nonumber
&\lesssim \sum_{p\in\mathfrak{N}\backslash\mathfrak{N}^{fC}}\int_{\mathfrak{e}}(\bar{r}(\varphi_{\mathfrak{m}}))(\bar{r}(\varphi_{\mathfrak{m}}))\theta_{\mathfrak{e}}\phi_p\\ \nonumber
& = \sum_{p\in\mathfrak{N}\backslash\mathfrak{N}^{fC}}\int_{\mathfrak{e}}(r(\varphi_{\mathfrak{m}}))(\bar{r}(\varphi_{\mathfrak{m}}))\theta_{\mathfrak{e}}\phi_p + \sum_{p\in\mathfrak{N}\backslash\mathfrak{N}^{fC}}\int_{\mathfrak{e}}(\bar{r}(\varphi_{\mathfrak{m}})-r(\varphi_{\mathfrak{m}}))(\bar{r}(\varphi_{\mathfrak{m}}))\theta_{\mathfrak{e}}\phi_p\\ \nonumber
& = \left<G_{\mathfrak{m}},\bar{r}(\varphi_{\mathfrak{m}})\theta_{\mathfrak{e}}\right>_{-1,1} - \sum_{p\in\mathfrak{N}\backslash\mathfrak{N}^{fC}}\int_{\gamma^I_p}\epsilon^2[\nabla \varphi_{\mathfrak{m}}]\bar{r}(\varphi_{\mathfrak{m}})\theta_{\mathfrak{e}}\phi_p \\ \nonumber
&\quad + \sum_{p\in\mathfrak{N}\backslash\mathfrak{N}^{fC}}\int_{\mathfrak{e}}(\bar{r}(\varphi_{\mathfrak{m}})-r(\varphi_{\mathfrak{m}}))(\bar{r}(\varphi_{\mathfrak{m}}))\theta_{\mathfrak{e}}\phi_p + \sum_{p\in\mathfrak{N}\backslash\mathfrak{N}^{fC}}\left<\lambda_{\mathfrak{m}} ,\phi_p \right>_{-1,1}c_p(\bar{r}(\varphi_{\mathfrak{m}})\theta_{\mathfrak{e}})\\ \label{LowerBoundInnerRes}
& \lesssim \|G_{\mathfrak{m}}\|_{\ast,\epsilon,\omega_p} \|\bar{r}(\varphi_{\mathfrak{m}})\theta_{\mathfrak{e}}\|_{\epsilon,\omega_p} + \|\bar{r}(\varphi_{\mathfrak{m}})- r(\varphi_{\mathfrak{m}})\|_{\mathfrak{e}}\|\bar{r}(\varphi_{\mathfrak{m}})\theta_{\mathfrak{e}}\|_{\mathfrak{e}}
\end{align}
as $c_p(\theta_{\mathfrak{e}})=0$ for all $p\in\mathfrak{N}^{sC}$ following from (\ref{MeanValueZero1}), $\left<\lambda_{\mathfrak{m}}, \phi_p\right>_{-1,1}=0$ for all $p\in\mathfrak{N}\backslash\mathfrak{N}^{C}$ and $\theta_{\mathfrak{e}}$ vanishes on the edges.

Exploiting (\ref{WeightedInverseInequality}) for $\theta_{\mathfrak{e}}$ instead of $\Psi_{\mathfrak{e}}$, dividing (\ref{LowerBoundInnerRes}) by $\|\bar{r}(\varphi_{\mathfrak{m}})\|_{\mathfrak{e}}$ and multiplying with $\mathrm{min}\{\frac{h_{\mathfrak{e}}}{\epsilon}, 1\} = (\mathrm{max}\{\frac{\epsilon}{h_{\mathfrak{e}}},1\})^{-1}$ we arrive at
\begin{align}\label{LowerBound1}
\mathrm{min}\{\frac{h_{\mathfrak{e}}}{\epsilon},1\}\|\bar{r}(\varphi_{\mathfrak{m}})\|_{\omega_p} \lesssim
\|G_{\mathfrak{m}}\|_{\ast,\epsilon,\omega_p} +\mathrm{min}\{\frac{h_{\mathfrak{e}}}{\epsilon},1\} \|\bar{r}(\varphi_{\mathfrak{m}})- r(\varphi_{\mathfrak{m}})\|_{\omega_p}.
\end{align}
Next, we apply the triangle inequality $\|r\|\leq \|\bar{r}\|+ \|\bar{r}-r\|$, exploit the definition of $r$, and $\bar{r}$, respectively, and together  with (\ref{LowerBound0}) we get the desired result 
\begin{align}\nonumber
\eta_{1,p} = \mathrm{min}\{\frac{h_{\mathfrak{e}}}{\epsilon},1\}\|r(\varphi_{\mathfrak{m}})\|_{\omega_p} &\lesssim 
\|G_{\mathfrak{m}}\|_{\ast,\epsilon,\omega_p} +\mathrm{osc}_p(f)\\ \label{LowerBoundEta1}
&\lesssim \|\varphi-\varphi_{\mathfrak{m}}\|_{\epsilon,\omega_p} + \|\lambda-\tilde{\lambda}_{\mathfrak{m}}\|_{\ast,\epsilon,\omega_p} +\mathrm{osc}_p(f).
\end{align}

In order to prove the lower bound in terms of $\eta_{2,p}$ we use the properties of side bubble functions. Following the ansatz given in  \cite{Verfuerth_1998b} we define side bubble functions with the help of basis functions belonging to a modified element. On the reference element $\hat{\mathfrak{e}}$ the corresponding transformation $\Phi_{\delta}:\mathbb{R}^n\rightarrow\mathbb{R}^n$ maps the coordinates $x_1,\ldots, x_n$ to $x_1,\ldots ,\delta x_n$.  The basis functions on the transformed reference element are given by $\hat{\phi}_{\delta,\hat{p}_i} := \hat{\phi}_{\hat{p}_i}\circ\Phi^{-1}_{\delta}$ for $i=1,\ldots, n+1$ on $\Phi_{\delta}(\hat{\mathfrak{e}})$ and $\hat{\phi}_{\delta,\hat{p}_i} = 0$ on $\hat{\mathfrak{e}}\backslash \Phi_{\delta}(\hat{\mathfrak{e}})$, e.g. in the two-dimensional case 
they are given by 
\begin{align*}
\hat{\phi}_{\delta,\hat{p}_1} = x_1,\qquad\hat{\phi}_{\delta,\hat{p}_2} = \frac{1}{\delta}x_2,\qquad \hat{\phi}_{\delta,\hat{p}_3} = (1-x_1-\frac{1}{\delta}x_2).
\end{align*}
Let $\hat{\mathfrak{s}}_i$ be the sides opposite to the nodes $\hat{p}_i$.
The modified side bubble function we will consider is given by $\Psi_{\delta,\hat{\mathfrak{s}}_n} := \hat{\phi}_{\delta,\hat{p}_{n+1}}\Pi_{i=1}^{n-1}\hat{\phi}_{\delta,\hat{p}_i}$, e.g. in the two-dimensional case 
$\Psi_{\delta,\hat{\mathfrak{s}}_3} = (1-x_1-\frac{1}{\delta}x_2)x_1$. 
Let $F_{\mathfrak{s}}:\mathfrak{e}\rightarrow\hat{\mathfrak{e}}$ be the linear transformation which maps $\mathfrak{s}$ on $\hat{\mathfrak{s}}_n$. 
Then it follows from \cite[Lemma 3.4]{Verfuerth_1998b} together with the transformation rule 
\begin{equation}\label{PropModBubbFunc_Trans}
\begin{split}
\|\Psi_{\delta,\mathfrak{s}}w\|_{\mathfrak{e}}&\lesssim h_{\mathfrak{e}}^{\frac{1}{2}}\sqrt{\delta}\|w\|_{\mathfrak{s}},\\
\|\frac{\partial}{\partial x_i}(\Psi_{\delta,\mathfrak{s}}w)\|_{\mathfrak{e}}&\lesssim h_{\mathfrak{e}}^{-\frac{1}{2}} \sqrt{\delta}\|w\|_{\mathfrak{s}},\quad\forall 1\leq i\leq n-1\\
\|\frac{\partial}{\partial x_n}(\Psi_{\delta,\mathfrak{s}}w)\|_{\mathfrak{e}}&\lesssim h_{\mathfrak{e}}^{-\frac{1}{2}} \frac{1}{\sqrt{\delta}}\|w\|_{\mathfrak{s}}.
\end{split}
\end{equation}
With respect to the $\|\cdot\|_{\epsilon}$ norm we get
\begin{equation}\label{WeightedInverseInequalitySide}
\|\Psi_{\delta,\mathfrak{s}}w\|_{\epsilon,\mathfrak{e}}\lesssim (h_{\mathfrak{e}}^{\frac{1}{2}}\delta^{\frac{1}{2}} + \epsilon h_{\mathfrak{e}}^{-\frac{1}{2}}\delta^{-\frac{1}{2}})\|w \|_{\mathfrak{s}}.
\end{equation} 
Similar to the proof of the lower bound in terms of $\eta_{1,p}$ we construct a linear combination $\theta_{\delta,\mathfrak{s}} = \sum_{j} a_j\Psi_{\delta,j}$ of modified side bubble functions $\Psi_{\delta,\mathfrak{s}}$ with respect to all sides $\tilde{\mathfrak{s}}_j$ of the partition of $\mathfrak{s}$ such that $c_p(\theta_{\delta,\mathfrak{s}})=0$. Therefore we assume that $\tilde{\omega}_p$ is the patch around $p$ with respect to two uniform red-refinements. We choose $a_j=0$ for all sides $\tilde{\mathfrak{s}}_j$ containing a node $p\in\mathfrak{N}^{sC}$ such that 
\begin{equation} \label{MeanValueZero}
\int_{\tilde{\mathfrak{e}}_j}\theta_{ \delta,\mathfrak{s}}\phi_{p} = 0\quad\mbox{with}\quad p\in\tilde{\mathfrak{s}}_j. 
\end{equation}
The other coefficients of the linear combination are chosen such that the bubble function $\theta_{\delta,\mathfrak{s}}$ fulfills the following property
\begin{equation} \label{RelNormInt}
\int_{\mathfrak{s}}1=\sum_{p\in\mathfrak{N}\backslash\mathfrak{N}^{fC}}\int_{\mathfrak{s}} \theta_{\delta,\mathfrak{s}}\phi_p.
\end{equation}

We set $w:= \epsilon^2[\nabla \varphi_{\mathfrak{m}}]$. 
Thus, we apply  (\ref{PropModBubbFunc_Trans}), (\ref{WeightedInverseInequalitySide}). Together with (\ref{MeanValueZero}) and (\ref{RelNormInt}), we get
\begin{align*}
\|\epsilon^2[\nabla \varphi_{\mathfrak{m}}]\|_{\mathfrak{s}}^2 &= \sum_{p\in\mathfrak{N}\backslash\mathfrak{N}^{fC}}\int_{\mathfrak{s}} \epsilon^4[\nabla \varphi_{\mathfrak{m}}] [\nabla \varphi_{\mathfrak{m}}]\theta_{\delta,\mathfrak{s}}\phi_p\\
&\lesssim \left<G_{\mathfrak{m}}, \epsilon^2[\nabla \varphi_{\mathfrak{m}}]\theta_{\delta,\mathfrak{s}}\right>_{-1,1} + \sum_{p\in\mathfrak{N}\backslash\mathfrak{N}^{fC}}\int_{\omega_{\mathfrak{s}}}r(\varphi_{\mathfrak{m}})\epsilon^2[\nabla \varphi_{\mathfrak{m}}]\theta_{\delta,\mathfrak{s}}\phi_p\\
&\qquad\quad+ \sum_{p\in\mathfrak{N}\backslash\mathfrak{N}^{fC}}\left<\tilde{\lambda}_{\mathfrak{m}} ,\phi_p \right>_{-1,1}c_p(\bar{r}(\varphi_{\mathfrak{m}})\theta_{\delta,\mathfrak{s}})\\
&\lesssim \|G_{\mathfrak{m}}\|_{\ast,\epsilon}\|\epsilon^2[\nabla \varphi_{\mathfrak{m}}]\theta_{\delta,\mathfrak{s}}\|_{\epsilon, \omega_{\mathfrak{s}}} + \|r(\varphi_{\mathfrak{m}})\|_{\omega_{\mathfrak{s}}} \|\epsilon^2[\nabla \varphi_{\mathfrak{m}}]\theta_{\delta,\mathfrak{s}}\|_{\omega_{\mathfrak{s}}}\\
&\lesssim \|G_{\mathfrak{m}}\|_{\ast,\epsilon} \{h_{\mathfrak{s}}^{-\frac{1}{2}}\delta^{-\frac{1}{2}}\epsilon + \delta^{\frac{1}{2}} h_{\mathfrak{s}}^{\frac{1}{2}}\} \|\epsilon^2[\nabla \varphi_{\mathfrak{m}}]\|_{\mathfrak{s}}  + \delta^{\frac{1}{2}}h_{\mathfrak{s}}^{\frac{1}{2}}\|r(\varphi_{\mathfrak{m}})\|_{\omega_{\mathfrak{s}}}  \|\epsilon^2[\nabla \varphi_{\mathfrak{m}}]\|_{\mathfrak{s}} .
\end{align*}
Dividing by $\|\epsilon^2[\nabla \varphi_{\mathfrak{m}}]\|_{\mathfrak{s}} $ and multiplying with $\epsilon^{-\frac{1}{2}}\mathrm{min}\{\frac{h_{\mathfrak{s}}}{\epsilon},1\}^{\frac{1}{2}}$ and choosing $\delta = \mathrm{min}\{\frac{\epsilon}{h_{\mathfrak{s}}},1\}$ we get the factors 
\begin{align*}
&\epsilon^{-\frac{1}{2}}\mathrm{min}\{\frac{h_{\mathfrak{s}}}{\epsilon},1\}^{\frac{1}{2}}\{\epsilon h_{\mathfrak{s}}^{-\frac{1}{2}}\mathrm{min}\{\frac{\epsilon}{h_{\mathfrak{s}}},1\}^{-\frac{1}{2}} + \mathrm{min}\{\frac{\epsilon}{h_{\mathfrak{s}}},1\}^{\frac{1}{2}} h_{\mathfrak{s}}^{\frac{1}{2}}\}\\
=&\epsilon^{-\frac{1}{2}}\mathrm{min}\{\frac{h_{\mathfrak{s}}}{\epsilon},1\}^{\frac{1}{2}} \epsilon^{\frac{1}{2}}\mathrm{min}\{\frac{h_{\mathfrak{s}}}{\epsilon},1\}^{-\frac{1}{2}} + \epsilon^{-\frac{1}{2}}\mathrm{min}\{\frac{h_{\mathfrak{s}}}{\epsilon},1\}^{\frac{1}{2}}\epsilon^{\frac{1}{2}}\mathrm{min}\{\frac{h_{\mathfrak{s}}}{\epsilon},1\}^{\frac{1}{2}}\\
\lesssim& 1 
\end{align*}
and 
\begin{align*}
\epsilon^{-\frac{1}{2}}\mathrm{min}\{\frac{h_{\mathfrak{s}}}{\epsilon},1\}^{\frac{1}{2}}\mathrm{min}\{\frac{\epsilon}{h_{\mathfrak{s}}},1\}^{\frac{1}{2}}h_{\mathfrak{s}}^{\frac{1}{2}} =   \mathrm{min}\{\frac{h_{\mathfrak{s}}}{\epsilon},1\}.
\end{align*}
Thus, together with the estimate (\ref{LowerBoundEta1}) we arrive at
\begin{align*}
\epsilon^{-\frac{1}{2}}\mathrm{min}\{\frac{h_{\mathfrak{s}}}{\epsilon},1\}^{\frac{1}{2}}\|\epsilon^2[\nabla \varphi_{\mathfrak{m}}]\|_{\mathfrak{s}} &\lesssim \|G_{\mathfrak{m}}\|_{\ast,\epsilon} + \mathrm{min}\{\frac{h_{\mathfrak{s}}}{\epsilon},1\} \|\bar{r}(\varphi)- r(\varphi_{\mathfrak{m}})\|_{\omega_{s}}
\end{align*}
and get the local lower bound
\begin{align}\label{LowerBoundEta2}
\eta_{2,p} \lesssim \|\varphi-\varphi_{\mathfrak{m}}\|_{\epsilon,\omega_p} + \|\lambda-\tilde{\lambda}_{\mathfrak{m}}\|_{\ast,\epsilon,\omega_p} + 
\mathrm{osc}_p(f). 
\end{align}
To derive a local lower bound in terms of $\eta_{3,p}$ we can proceed in the same way to get
\begin{align}\label{LowerBoundEta3}
\eta_{3,p} \lesssim \|\varphi-\varphi_{\mathfrak{m}}\|_{\epsilon,\omega_p} + \|\lambda-\tilde{\lambda}_{\mathfrak{m}}\|_{\ast,\epsilon,\omega_p} + \mathrm{osc}_p(f) + \mathrm{osc}_p(\pi).
\end{align}

Theorem \ref{Theorem:LowerBound} follows from (\ref{LowerBoundEta1}, \ref{LowerBoundEta2}, \ref{LowerBoundEta3}).

\subsection{Local error bound in terms of $\eta_{4,p}$}

In order to show that also $\eta_{4,p}$ constitutes a local lower bound, we proceed almost as in \cite{Krause_Veeser_Walloth_2015}.
We assume $s_p>0$ and $(g_{\mathfrak{m}}-\varphi_{\mathfrak{m}})(q)\ge 0$ for at least one node in $\omega_p$ so that $\eta_{4,p}\neq 0$.
The assumption $s_p>0$ implies that $p$ is a contact node, i.e. $(g_{\mathfrak{m}}-\varphi_{\mathfrak{m}})(p)=0$.
Choose a node $\hat{q}$ in $\omega_p$ such that $(g_{\mathfrak{m}}-\varphi_{\mathfrak{m}})(\hat{q})\ge (g_{\mathfrak{m}}-\varphi_{\mathfrak{m}})(q)$ for all $q\in\omega_p$. We denote the unit vector pointing from $p$ to $\hat{q}$ by $\boldsymbol{\tau}$.
We denote the element to wich $p$ and $\hat{q}$ belong  by $\mathfrak{e}_1$ and the element in $\omega_p$ which is intersected by $-\boldsymbol{\tau}$, starting in $p$, is denoted by $\mathfrak{e}_N$. The elements between $\mathfrak{e}_1$ and $\mathfrak{e}_N$ are denoted in order by $\mathfrak{e}_i$, $i=2,\ldots, N-1$.
We use Taylor expansion around $(g_{\mathfrak{m}}-\varphi_{\mathfrak{m}})(p)=0$ and add the gradient in the opposite direction $-\boldsymbol{\tau}$. Due to the constraints, $\nabla|_{\mathfrak{e}}(g_{\mathfrak{m}}-\varphi_{\mathfrak{m}})\cdot (\pm \boldsymbol{\tau}) \ge 0$ holds and thus
\begin{align*}
(g_{\mathfrak{m}}-\varphi_{\mathfrak{m}})(\hat{q}) &= \nabla|_{\mathfrak{e}_1}(g_{\mathfrak{m}}-\varphi_{\mathfrak{m}})(\hat{q}-p)\lesssim h_p\nabla|_{\mathfrak{e}_1}(g_{\mathfrak{m}}-\varphi_{\mathfrak{m}})\cdot \boldsymbol{\tau}\\
& \lesssim h_p(\nabla|_{\mathfrak{e}_1}(g_{\mathfrak{m}}-\varphi_{\mathfrak{m}}) -  \nabla|_{\mathfrak{e}_N}(g_{\mathfrak{m}}-\varphi_{\mathfrak{m}}))\cdot \boldsymbol{\tau}\\
& \lesssim h_p\sum_{i=1}^N|\nabla|_{\mathfrak{e}_i} (g_{\mathfrak{m}}-\varphi_{\mathfrak{m}})-\nabla|_{\mathfrak{e}_{i-1}}(g_{\mathfrak{m}}-\varphi_{\mathfrak{m}})|\\
&\lesssim h_p h_p^{-\frac{d-1}{2}}\|[\nabla (g_{\mathfrak{m}}-\varphi_{\mathfrak{m}})]^I\|_{\gamma^I_p}.
\end{align*}
For the ease of presentation we set $v_{\mathfrak{m}}= (g_{\mathfrak{m}}-\varphi_{\mathfrak{m}})$ in the following.
Further, we exploit  
\begin{align*}
\left<\lambda_{\mathfrak{m}} ,\phi_p \right>:=\int_{\gamma_{p}^I}\epsilon^2[\nabla (\varphi_{\mathfrak{m}})]\phi_p +\int_{\gamma_{p}^N}(\pi-\epsilon^2\nabla \varphi_{\mathfrak{m}}\cdot\boldsymbol{n})\phi_p   + \int_{\omega_p}r(\varphi_{\mathfrak{m}})\phi_p. 
\end{align*}
Putting together and assuming $h_p <\epsilon$
\begin{align*}
&\eta^2_{4,p}\\
 &= \left<\lambda_{\mathfrak{m}} ,\phi_p \right>c_p(g_{\mathfrak{m}}-\varphi_{\mathfrak{m}})\\
 &\leq  h_p^{d}h_p h_p^{-\frac{d-1}{2}}\|[\nabla v_{\mathfrak{m}}]\|_{\gamma^I_p}h_p^{-d}\left(\|\epsilon^2[\nabla \varphi_{\mathfrak{m}}]\|_{\gamma^I_p} \|\phi_p\|_{\gamma^I_p}+ \|\pi-\epsilon^2\nabla \varphi_{\mathfrak{m}}\cdot\boldsymbol{n}\|_{\gamma^N_p} \|\phi_p\|_{\gamma^N_p}\right.\\
  &\left.\qquad+\|r(\varphi_{\mathfrak{m}})\|_{\omega_p}\|\phi_p\|_{\omega_p} \right)\\
&\leq h_p h_p^{-\frac{d-1}{2}}\|[\nabla v_{\mathfrak{m}}]\|_{\gamma^I_p}\left(\|\epsilon^2[\nabla \varphi_{\mathfrak{m}}]\|_{\gamma^I_p} h_p^{\frac{d-1}{2}}+\|\pi-\epsilon^2\nabla \varphi_{\mathfrak{m}}\cdot\boldsymbol{n}\|_{\gamma^N_p} h_p^{\frac{d-1}{2}}+  \|r(\varphi_{\mathfrak{m}})\|_{\omega_p}h_p^{\frac{d}{2}} \right)\\
&\leq \frac{h_p^{\frac{1}{2}}}{\epsilon}\|\epsilon^2[\nabla \varphi_{\mathfrak{m}}]\|_{\gamma^I_p}\frac{h_p^{\frac{1}{2}}}{\epsilon}\|\epsilon^2[\nabla v_{\mathfrak{m}}]\|_{\gamma^I_p} +
\frac{h_p^{\frac{1}{2}}}{\epsilon}\|\pi-\epsilon^2\nabla \varphi_{\mathfrak{m}}\cdot{\boldsymbol{n}}\|_{\gamma^N_p}\frac{h_p^{\frac{1}{2}}}{\epsilon}\|\epsilon^2[\nabla v_{\mathfrak{m}}]\|_{\gamma^I_p} \\
&\qquad + \frac{h_p^{\frac{1}{2}}}{\epsilon}\|\epsilon^2[\nabla v_{\mathfrak{m}}]\|_{\gamma^I_p}\frac{h_p}{\epsilon}\|r(\varphi_{\mathfrak{m}})\|_{\omega_p}\\
&\lesssim\frac{h_p}{\epsilon}\frac{1}{\epsilon}\|\epsilon^2[\nabla \varphi_{\mathfrak{m}}]\|^2_{\gamma^I_p}  +\frac{h_p}{\epsilon}\frac{1}{\epsilon}\|\pi-\epsilon^2\nabla \varphi_{\mathfrak{m}}\cdot\boldsymbol{n}\|^2_{\gamma^N_p}  + \frac{h_p^2}{\epsilon^2}\|r(\varphi_{\mathfrak{m}})\|^2_{\omega_p}  + \frac{h_p}{\epsilon}\frac{1}{\epsilon}\|\epsilon^2[\nabla v_{\mathfrak{m}}]\|^2_{\gamma^I_p}\\
&\lesssim \eta^2_{1,p} + \eta^2_{2,p} + \eta^2_{3,p} + \frac{h_p}{\epsilon}\frac{1}{\epsilon}\|\epsilon^2[\nabla v_{\mathfrak{m}}]\|^2_{\gamma^I_p}\\
&\lesssim \eta^2_{1,p} + \eta^2_{2,p} + \eta^2_{3,p} + \frac{h_p}{\epsilon}\frac{1}{\epsilon}\|\epsilon^2[\nabla g_{\mathfrak{m}}]\|^2_{\gamma^I_p}
\end{align*}
where in the last line we exploited the definition of $v_{\mathfrak{m}}$ and the triangle inequality.
Thus, together with (\ref{LowerBoundEta1}), (\ref{LowerBoundEta2}), (\ref{LowerBoundEta3}) we get 
\begin{equation}\label{ProofEta4Robust}
\eta^2_{4,p} \lesssim  \|\varphi-\varphi_{\mathfrak{m}}\|_{\epsilon,\omega_p} + \|\lambda-\tilde{\lambda}_{\mathfrak{m}}\|_{\ast,\epsilon,\omega_p} + \mathrm{osc}_p(f) + \mathrm{osc}_p(\pi) +  \left(\frac{h_p}{\epsilon}\frac{1}{\epsilon}\|\epsilon^2[\nabla g_{\mathfrak{m}}]\|^2_{\gamma^I_p}.\right)^{\frac{1}{2}}.
\end{equation}
In the remaining case $h_p >\epsilon$ the weightings in $\eta_{2,p},\eta_{3,p}$ are $\mathrm{min}\{\frac{h_p}{\epsilon},1\}^{\frac{1}{2}}\frac{1}{\epsilon^{\frac{1}{2}}} =\epsilon^{-\frac{1}{2}}$ and  in $\eta_{1,p}$ it is $\mathrm{min}\{\frac{h_p}{\epsilon},1\}=1$. We exploit $h_p\leq 1$ and proceed as before 
\begin{align*}
\eta^2_{4,p} &= \left<\lambda_{\mathfrak{m}} ,\phi_p \right>c_p(g_{\mathfrak{m}}-\varphi_{\mathfrak{m}})\\
&\leq \frac{h_p^{\frac{1}{2}}}{\epsilon}\|\epsilon^2[\nabla \varphi_{\mathfrak{m}}]\|_{\gamma^I_p}\frac{h_p^{\frac{1}{2}}}{\epsilon}\|\epsilon^2[\nabla v_{\mathfrak{m}}]\|_{\gamma^I_p} +\frac{h_p^{\frac{1}{2}}}{\epsilon}\|\pi-\epsilon^2\nabla \varphi_{\mathfrak{m}}\cdot\boldsymbol{n}\|_{\gamma^N_p}\frac{h_p^{\frac{1}{2}}}{\epsilon}\|\epsilon^2[\nabla v_{\mathfrak{m}}]\|_{\gamma^I_p} \\
&\qquad+ \frac{h_p^{\frac{1}{2}}}{\epsilon}\|\epsilon^2[\nabla v_{\mathfrak{m}}]\|_{\gamma^I_p}\frac{h_p}{\epsilon}\|r(\varphi_{\mathfrak{m}})\|_{\omega_p}\\
&\leq \epsilon^{-\frac{1}{2}}\|\epsilon^2[\nabla \varphi_{\mathfrak{m}}]\|_{\gamma^I_p}\epsilon^{-\frac{3}{2}}\|\epsilon^2[\nabla v_{\mathfrak{m}}]\|_{\gamma^I_p} + \epsilon^{-\frac{1}{2}}\|\pi-\epsilon^2\nabla \varphi_{\mathfrak{m}}\cdot\boldsymbol{n}\|_{\gamma^N_p}\epsilon^{-\frac{3}{2}}\|\epsilon^2[\nabla v_{\mathfrak{m}}]\|_{\gamma^I_p}\\
&\qquad + \epsilon^{-2}\|\epsilon^2[\nabla v_{\mathfrak{m}}]\|_{\gamma^I_p}\|r(\varphi_{\mathfrak{m}})\|_{\omega_p}\\
&\lesssim \eta^2_{1,p} + \eta^2_{2,p} + \eta^2_{3,p} + \epsilon^{-4}\|\epsilon^2[\nabla v_{\mathfrak{m}}]\|^2_{\gamma^I_p}\\
&\lesssim \eta^2_{1,p} + C(\epsilon^{-3})\eta^2_{2,p} + \eta^2_{3,p} + \epsilon^{-4}\|\epsilon^2[\nabla g_{\mathfrak{m}}]\|^2_{\gamma^I_p}
\end{align*}
where in the last line we exploited the definition of $v_{\mathfrak{m}}$ and the triangle inequality.
Thus, together with (\ref{LowerBoundEta1}), (\ref{LowerBoundEta2}), (\ref{LowerBoundEta3}) we get 
\begin{equation}\label{ProofEta4NonRobust}
\eta_{4,p} \lesssim C(\epsilon^{-\frac{3}{2}})\left( \|\varphi-\varphi_{\mathfrak{m}}\|_{\epsilon,\omega_p} + \|\lambda-\tilde{\lambda}_{\mathfrak{m}}\|_{\ast,\epsilon,\omega_p} + \mathrm{osc}_p(f) + \mathrm{osc}_p(\pi) +  \epsilon^{-\frac{1}{2}}\left(\|\epsilon^2[\nabla g_{\mathfrak{m}}]\|^2_{\gamma^I_p}\right)^{\frac{1}{2}}\right).
\end{equation}

Theorem \ref{Theorem:LowerBound2} follows from (\ref{LowerBound0}), (\ref{ProofEta4NonRobust}) and (\ref{ProofEta4Robust}).

\section{Numerical results}\label{Sec:NumRes}

The implementation has been carried out in {\tt MATLAB}. 
As basis for the implementation of the adaptive mesh generation we have used \cite[Chapter 5]{Funken_Praetorius_Wissgott_2011} as well as \cite{Bartels_Schreier_2012}. As solver for the variational inequalities we implemented a primal-dual-active set method similar to \cite[Chapter 5.3.1]{Bartels_2015}.

We consider two different examples in $2D$. The starting grid has been four times uniformly refined in Example 1 and three times uniformly refined in Example 2 by means of newest vertex bisection, compare \cite[Chapter 5]{Funken_Praetorius_Wissgott_2011}. As marking strategy for the adaptive process we use the mean value strategy, i.e. an element is marked for refinement if its local element estimator is bigger than $1.2$ times the mean value of all element estimators. The maximal number of elements which has to be passed before the refinement process stops is set to $20000$ elements in both Examples. 

For the first example we define the rotation matrix $\boldsymbol{R}:= \left(\begin{matrix}\cos{\alpha} & -\sin{\alpha}\\ \sin{\alpha} &\cos{\alpha} \end{matrix}\right)$ with $\alpha = \frac{-\pi}{16}$ and the rotated strips $I_1:= \boldsymbol{R}\left[\left(\begin{matrix} -1.5\\ y\end{matrix}\right),\left(\begin{matrix}-0.5\\ y\end{matrix}\right) \right] $, $I_2 := \boldsymbol{R}\left[\left(\begin{matrix} 0.5\\ y\end{matrix}\right),\left(\begin{matrix}1.5\\ y\end{matrix}\right) \right]$ for all $y\in \mathbb{R}$. We define the domains $\Omega:= [-2.5, 2.5]\times[-2.5,2.5]$ and $\Omega_1:=\Omega\cap I_1$ and $\Omega_2:=\Omega\cap I_2$. 
The first problem is given by
\begin{example}[Boundary layers enforced by obstacle constraints]\label{Example1}
\begin{align*}
-\epsilon^2 \Delta \varphi + \varphi &= -1\quad \mbox{in} \quad\Omega_1\cup\Omega_2\\
-\epsilon^2 \Delta \varphi + \varphi &\ge -1\quad \mbox{in} \quad \Omega\backslash(\Omega_1\cup\Omega_2)\\
\varphi&\ge 0 \quad \mbox{in}\quad\Omega\backslash \left(\Omega_1\cup\Omega_2\right)\\
(\varphi)(-\epsilon^2\Delta \varphi+  \varphi +1 ) &= 0 \quad \mbox{in}\quad\Omega\backslash \left(\Omega_1\cup\Omega_2\right)\\
\boldsymbol{\nabla}\varphi\cdot\boldsymbol{n} &= 0\quad \mbox{on}\quad\Gamma
\end{align*}
\end{example}
We remark that the solution can be explicitly computed on the non-rotated strips $\tilde{I}_1:= [-1.5,-0.5]\times[-2.5,2.5]$, $\tilde{I}_2 := [0.5,1.5]\times[-2.5,2.5]$ with the ansatz $\tilde{\varphi}_1(x) = c_1\exp(\frac{x}{\epsilon}) + c_2\exp(\frac{-x}{\epsilon})-1$ and the boundary conditions $\tilde{\varphi}_1(-1.5,y)=0$, $\tilde{\varphi}_1(-0.5,y)=0$  on $\tilde{I}_1$ and with the ansatz $\tilde{\varphi}_2(x) = c_3\exp(\frac{x}{\epsilon}) + c_4\exp(\frac{-x}{\epsilon})-1$ and the boundary conditions $\tilde{\varphi}_2(0.5,y)=0$, $\tilde{\varphi}_2(1.5,y)=0$ on $\tilde{I}_2$. Further we define $\tilde{\varphi}_3= 0$ on the rest of the domain $\Omega\backslash\tilde{I}_1\cup \tilde{I}_2$.
Thus, by defining 
\begin{align*}
\tilde{\varphi}:=\left\{\begin{array}{cc} 
\tilde{\varphi}_1 & \mbox{in} \quad \tilde{I}_1 \\ 
 \tilde{\varphi}_2 & \mbox{in} \quad \tilde{I}_2\\
 \tilde{\varphi}_3 & \mbox{in} \quad \Omega\backslash\left(\tilde{I}_1\cap\tilde{I}_2\right)
\end{array}\right.
\end{align*}
the solution of Example \ref{Example1} is given by 
\begin{align*}
\varphi = \tilde{\varphi}\circ\boldsymbol{R}^{-1} \quad\mbox{in}\quad\Omega.
\end{align*}
\begin{figure}%[H]
\centering
\subfloat[{\it solution}]{\includegraphics[width=0.7\textwidth]{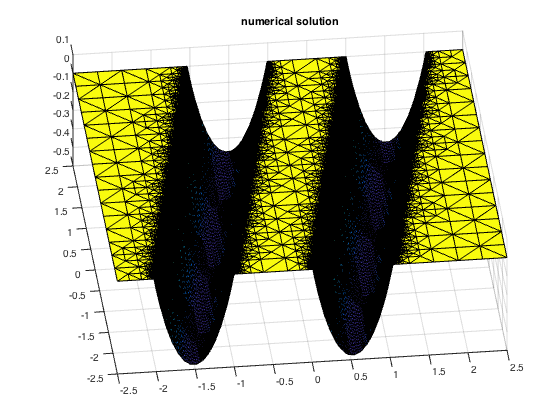}}\\
\subfloat[{\it adaptively refined mesh steered by the presented new estimator $\eta$}]{\includegraphics[width=0.47\textwidth]{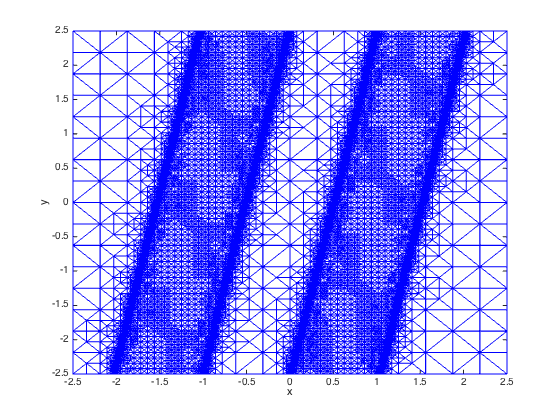}}\hspace{1em}
\subfloat[{\it adaptively refined mesh steered by the standard estimator $\eta^{std}$}]{\includegraphics[width=0.47\textwidth]{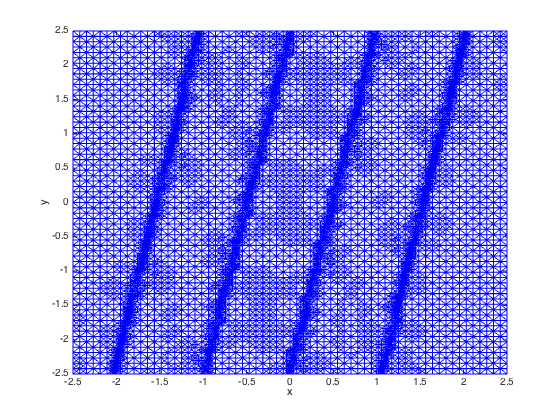}}
\caption{Example 1 with $\epsilon=0.4$}
\label{fig:deformation_bdLayer}
\end{figure}
\begin{figure}
\begin{center}
\subfloat[{\it convergence of error for $\eta^{std}$ and $\eta$ ($\epsilon = 0.08$)}]{\includegraphics[width=0.5\textwidth]{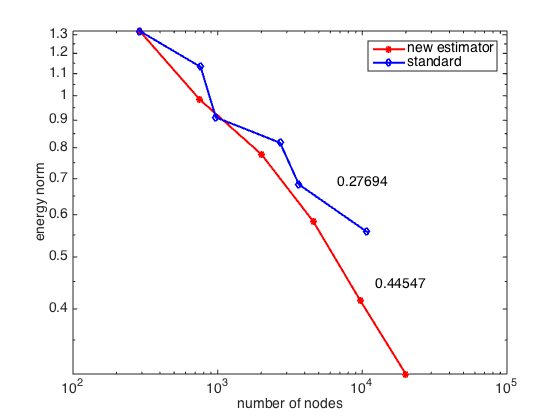}}\\
\subfloat[{\it efficiency index for $\eta^{nr}$ (continuous lines) and $\eta$ (dashed lines)}]{\includegraphics[width=0.45\textwidth]{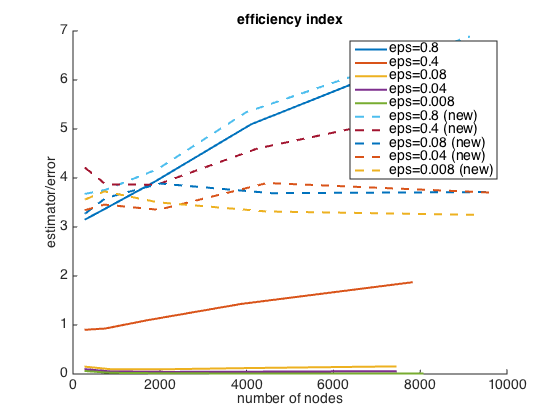}}\hspace{1em}
\subfloat[{\it efficiency index for $\eta^{nr}$ and $\eta$ in last refinement step}]{\includegraphics[width=0.45\textwidth]{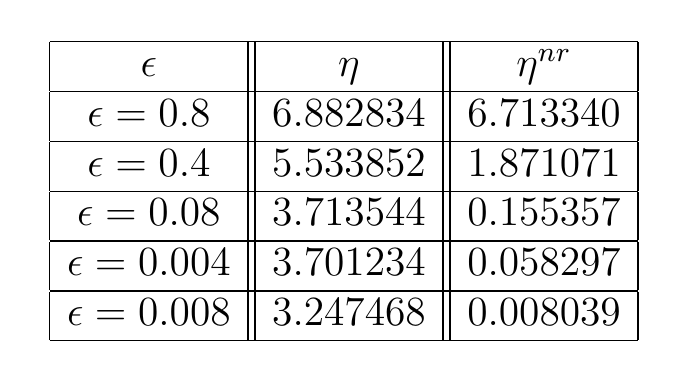}}
\end{center}
\caption{Example 1: Experimental order of convergence and efficiency index}
\label{fig:efficiency_bdLayer}
\end{figure}

In Figure \ref{fig:deformation_bdLayer}(a) the solution for $\epsilon=0.4$ is plotted on the adaptively refined grid. In the following two figures we plotted the adaptively refined mesh steered by the presented new estimator $\eta$ (\ref{Def_Estimator}) in Figure \ref{fig:deformation_bdLayer}(b) and steered by the standard residual estimator (\ref{StdResEst}) for linear elliptic problems without constraints in Figure \ref{fig:deformation_bdLayer}(c). 
In Figure \ref{fig:deformation_bdLayer}(c) also the area of full-contact is well-resolved such that one can see no clear difference between the area of contact and the area where $\varphi_1$ or $\varphi_2$ is the solution. In contrast in Figure \ref{fig:deformation_bdLayer}(b) the area of full-contact where the solution $\varphi_3=0$ is fixed to the obstacle the refinement is less strong. It is obvious that the presented new estimator gives rise to a good resolution of the free boundary, the critical region between the areas of contact and no-contact, and avoids over-refinement in the area of full-contact. 

As we can compute an exact reference solution we plot the error reduction  in the energy norm $\|\varphi-\varphi_{\mathfrak{m}}\|_{\epsilon}$ for the choice of $\epsilon=0.08$ in Figure \ref{fig:efficiency_bdLayer}(a) with logarithmic scales on both axes. The experimental order of convergence is lower for the standard residual estimator  (\ref{StdResEst}) compared to the new estimator $\eta$ (\ref{Def_Estimator}).

Additionally to prove that our estimator is not only reliable and efficient we show the robustness in Figure \ref{fig:efficiency_bdLayer}(b). Therefore we plotted the efficiency index against the number of nodes for different choices of $\epsilon\leq1$. The dashed lines refer to the presented new estimator $\eta$, i.e. the efficiency index is given by $\frac{\eta}{\|\varphi-\varphi_{\mathfrak{m}}\|_{\epsilon}}$ and the continuous lines refer to the estimator $\eta^{nr}$, defined at the end of Section \ref{Subsec:EstimatorMainRes}, i.e. the efficiency index is given by $\frac{\eta^{nr}}{\|\varphi-\varphi_{\mathfrak{m}}\|_{1}}$. The efficiency index of the estimator $\eta$ stays in the same range for different choices of $\epsilon$ while it decreases significantly for $\eta^{nr}$. That the efficiency index decreases with the order of $\epsilon$ for the estimator $\eta^{nr}$ can be seen clearer in the tabular in Figure \ref{fig:efficiency_bdLayer}(c) where the efficiency index is listed for the last refinement step for both estimators and different choices of $\epsilon$.

\begin{figure}%[H]
\centering
\subfloat[{\it solution}]{\includegraphics[width=0.7\textwidth]{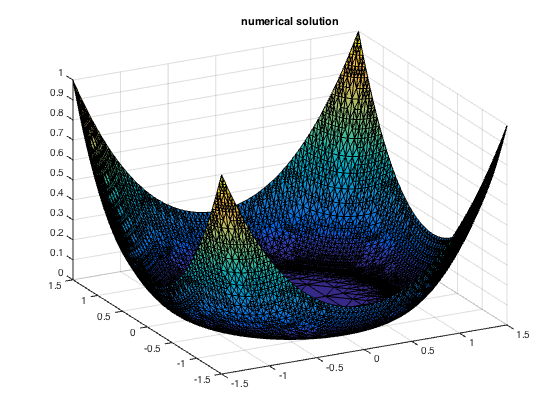}}\\
\subfloat[{\it adaptively refined mesh steered by the presented new estimator $\eta$}]{\includegraphics[width=0.47\textwidth]{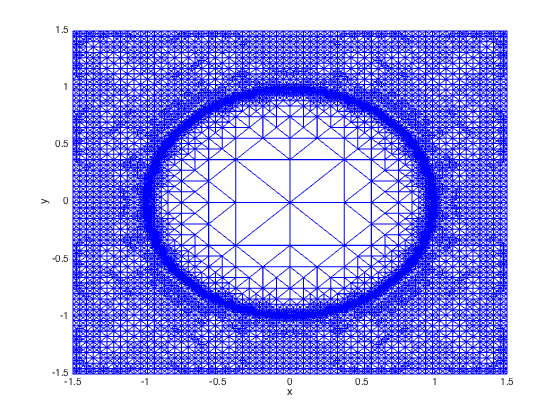}}\hspace{1em}
\subfloat[{\it adaptively refined mesh steered by the standard estimator $\eta^{std}$}]{\includegraphics[width=0.47\textwidth]{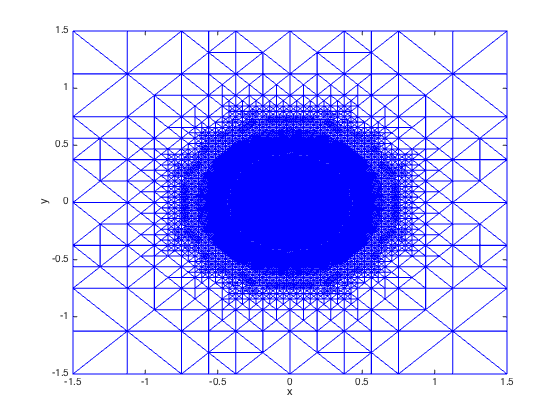}}
\caption{Example 2 with $\epsilon=0.01$}
\label{fig:deformation_smoothEx}
\end{figure}
\begin{figure}
\subfloat[{\it efficiency index for $\eta^{nr}$ (continuous lines) and $\eta$ (dashed lines)}]{\includegraphics[width=0.45\textwidth]{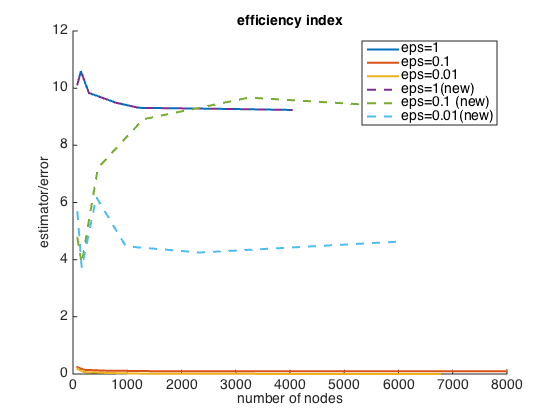}}\hspace{1em}
\subfloat[{\it efficiency index for $\eta^{nr}$ and $\eta$ in last refinement step}]{\includegraphics[width=0.45\textwidth]{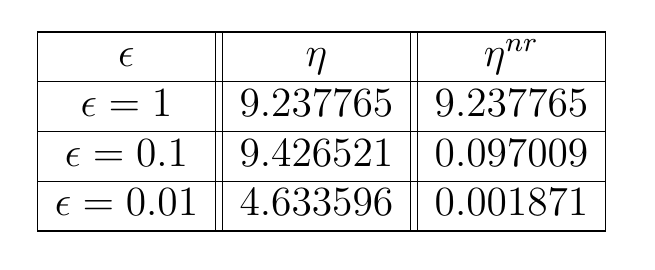}}
\caption{Example 2: Efficiency index}
\label{fig:efficiency_smoothEx}
\end{figure}

To show that the estimator enables a good resolution of the free boundary even though there is no boundary layer enforced by the constraints as in Example 1, we give a second example which has a smooth solution $\varphi$. Therefore, we adapt the example 5.1 of \cite{Bartels_Carstensen_2004}. Let $\Omega = [-1,1]\times[-1,1]$ and the radius $r:=\left|\left(\begin{matrix}x\\y\end{matrix}\right)\right|_2$ where $|\cdot|_2$ is the Euclidean norm. 
\begin{example}[Smooth example]\label{Example2}
\begin{align*}
-\epsilon^2 \Delta \varphi + \varphi &\ge f\quad \mbox{in}\quad\Omega\\
 \varphi&\ge 0 \quad \mbox{in}\quad\Omega\\
 (\varphi)(-\epsilon^2\Delta \varphi+  \varphi -f ) &= 0\quad \mbox{in}\quad \Omega\\
\varphi_D &=  \frac{r^2}{2}-\mathrm{ln}(r)-\frac{1}{2}\quad \mbox{on}\quad\Gamma
\end{align*}
with 
\begin{align*}
f = \left\{\begin{array}{cc} 
-2\epsilon^2  + \frac{r^2}{2}-\mathrm{ln}(r)-\frac{1}{2}\quad &\mbox{in} \quad r\ge 1\\
-2\epsilon^2  + \frac{1}{2}\left(r^2-1\right)\quad &\mbox{in} \quad r\leq 1
\end{array}\right. .
\end{align*}
\end{example} 
The force $f$ is constructed in such a way that penetration has to be avoided by enforcing the constraints. 

In Figure \ref{fig:deformation_smoothEx} we show  the solution on the adaptively refined grid. In the following two pictures we show the adaptively refined grid steered by the presented new estimator $\eta$  (\ref{Def_Estimator}) in Figure \ref{fig:deformation_smoothEx}(b) and steered by the standard estimator $\eta^{std}$ (\ref{StdResEst}) in \ref{fig:deformation_smoothEx}(c). Even though the solution is smooth at the transition zone, the free boundary is well resolved while the area of full-contact is not over-refined in Figure \ref{fig:deformation_smoothEx}(b). In contrast in Figure \ref{fig:deformation_smoothEx}(c) the strongest refinement has been taken place in the area of full-contact where $\varphi_{\mathfrak{m}}=g_{\mathfrak{m}}$. Further, we can see the efficiency index for $\eta$ and $\eta^{nr}$ in Figure \ref{fig:efficiency_smoothEx}(a) and again for the last refinement step in the tabular in Figure  \ref{fig:efficiency_smoothEx}(b). The efficiency index stays in the same range for the presented new estimator $\eta$.

\section*{Acknowledgments}
The author would like to thank Professor Winnifried Wollner for fruitful discussions on the topic of this work.

This work was funded by the Deutsche Forschungsgemeinschaft
  (DFG, German Research Foundation) – Projektnummer 392587580.

\end{document}